\documentclass[11pt,twoside,reqno]{amsart}

\usepackage{lmodern}
\usepackage[utf8]{inputenc}
\usepackage{amsmath}
\usepackage{graphicx}
\usepackage{amssymb}
\usepackage{esint}
\usepackage[dvipsnames]{xcolor}
\usepackage{tikz}
\usepackage{xxcolor}
\usepackage{floatrow}
\usepackage{color}
\usepackage{amsthm}
\usepackage{epsfig}
\usepackage[english]{babel}
\usepackage{hyperref}
\usepackage[normalem]{ulem}
\usepackage{mathrsfs}
\usepackage{tikz}
\usetikzlibrary{decorations.markings,backgrounds}
\usetikzlibrary{arrows.meta}
\usepackage{pgfplots}
\usepackage{subfigure}
\usepackage{caption}
\usepackage{bbm}

\usepackage{stmaryrd}

\usepackage{float}
\usepackage{pst-plot}

\hypersetup{
	colorlinks,
	linkcolor={red!80!black},
	citecolor={blue!50!black},
	urlcolor={blue!80!black}
}
 \usepackage{diagbox}

\usepackage[a4paper,top=3.5 cm,bottom=3 cm,left=2.5 cm,right=2.5 cm]{geometry}
\usepackage{dsfont}
\usepackage{mathtools}

\allowdisplaybreaks
\usepackage[colorinlistoftodos]{todonotes}
\usepackage{url}

\usepackage{graphicx,tikz}

\newtheorem{theorem}{Theorem}[section]
\newtheorem{lemma}{Lemma}[section]
\newtheorem{corollary}{Corollary}[section]
\newtheorem{definition}{Definition}[section]

\numberwithin{equation}{section}
\begin{document}
\begin{abstract}
We investigate heteroclinical solutions of a vector-valued Bose-Einstein condensation system involving the $p$-Laplacian. The main difficulty comes from the degeneracy of the $p$-Laplacian and the possible nonsmooth behavior of the potential wells.

Under suitable assumptions on the double-well potential, we establish the existence and detailed asymptotic behavior of heteroclinical solutions. In particular, we prove the strict monotonicity of every component, classify the blow-up profiles near the potential wells through Weiss-type monotonicity formulas, and obtain an almost homogeneity property of the solution.

Furthermore, after re-parametrizing the heteroclinical solution by its $l_p$-arc length, we derive a curvature estimate for the trajectory near the potential wells. These results provide the geometric control needed for the construction of radial barrier functions for the corresponding Bose-Einstein condensation system.
\end{abstract}

\author{Leyun Wu}
\address{School of Mathematics, South China University of Technology, Guangzhou, 510640, P. R. China}\email{leyunwu@scut.edu.cn}
\author{Chilin Zhang}
\address{School of Mathematical Sciences, Fudan University, Shanghai 200433, P. R. China}\email{zhangchilin@fudan.edu.cn}

\title[Curvature estimate]{Curvature estimate for the heteroclinical solution to a Bose-Einstein condensation system}

\maketitle

\maketitle

2020 {\bf MSC.} 35Q56; 34D05; 35R35.


\section{Introduction}
Phase transition phenomena arise in a wide variety of physical systems and have been extensively studied from both physical and mathematical perspectives. Among various models describing phase transitions, the Allen-Cahn equation and its vector-valued generalizations play a fundamental role in describing interfaces between different stable phases. In particular, heteroclinic connections, which represent one-dimensional transition layers connecting different equilibrium states, provide an important framework for understanding the geometry and structure of phase boundaries.

A particularly important class of vector-valued phase transition models comes from multi-component Bose-Einstein condensates. In the strongly segregated case, different condensate components tend to separate into distinct phases, and the transition interfaces between these phases are described by coupled nonlinear elliptic systems. The mathematical analysis of such interfaces involves several challenging issues, including the interaction between different components, the geometry of the potential wells, and the possible degeneracy of the underlying diffusion operators.

In this paper, we investigate the asymptotic behavior of heteroclinical solutions to the following one-dimensional Bose-Einstein condensation system:
\begin{equation*}
    \Delta_{p}\mathbf{u}=DW(\mathbf{u}),\quad\mathbf{u}:\mathbb{R}\to[-\Lambda_{1},\Lambda_{1}]\times\cdots\times[-\Lambda_{m},\Lambda_{m}].
\end{equation*}
Here, $\mathbf{u}$ is an $m$-dimensional vector-valued function, which is written as
\begin{equation*}
    \mathbf{u}=(u^{1},\cdots,u^{m}),\quad\mbox{where }u^{i}\in[-\Lambda_{i},\Lambda_{i}]
\end{equation*}
For simplicity, we denote
\begin{equation*}
    \mathbf{e}=(\Lambda_{1},\cdots,\Lambda_{m}).
\end{equation*}

The potential function $W$ is assumed to be a double-well potential satisfying the following conditions:
\begin{itemize}
    \item[(A1)] For all $\mathbf{u}\neq\pm\mathbf{e}$, $W$ is $C^{2}$ near $\mathbf{u}$ and $D_{ij}W(\mathbf{u})<0$ for all $i\neq j$.
    \item[(A2)] $W(\mathbf{u})>0$ for all $\mathbf{u}\neq\pm\mathbf{e}$, and $W(\pm\mathbf{e})=0$.
    \item[(A3)] $W(\mathbf{u})=W_{-}(\mathbf{u}+\mathbf{e})\cdot\Big(1+E_{-}(\mathbf{u}+\mathbf{e})\Big)$ near $-\mathbf{e}$, and $W(\mathbf{u})=W_{+}(\mathbf{e}-\mathbf{u})\cdot\Big(1+E_{+}(\mathbf{e}-\mathbf{u})\Big)$ near $\mathbf{e}$.
    \item[(A4)] Here, $W_{\pm}:[0,+\infty)^{m}\to\mathbb{R}$ is a homogeneous function of order $q_{\pm}>0 $, respectively. We also assume that $W_{+}$ and $W_{-}$ are both $C^{2}([0,+\infty)^{m}\setminus\{0\})$, strictly positive, and satisfy $D_{ij}W_{\pm}<0$ on the sphere $[0,+\infty)^{m}\cap\partial B_{1}$.
    \item[(A5)] Here, $E_{\pm}:[0,\epsilon]^{m}\to\mathbb{R}$ are both $C^{2}([0,\epsilon]^{m})$ functions vanishing at the origin.
\end{itemize}
The operator $\Delta_{p}$ with $p>1$ denotes the standard $p$-Laplacian. In the one-dimensional setting, it takes the form
\begin{equation*}
    \Delta_{p}f(x)=\frac{d}{dx}\Big(|f'(x)|^{p-2}f'(x)\Big)=(p-1)|f'(x)|^{p-2}f''(x).
\end{equation*}
Consequently, the system takes the following precise form:
\begin{equation}\label{eq. GP equation}
    \Delta_{p}u_{i}=(p-1)|u^{i}_{x}|^{p-2}u^{i}_{xx}=W_{i}(u^{1},\cdots,u^{m}),\quad1\leq i\leq m.
\end{equation}

The classical Allen-Cahn equation corresponds to the scalar
case with quadratic diffusion. It consists of studying minimizers (or critical points) of the Ginzburg-Landau energy
\begin{equation*}
    J(u,\Omega)=\int_{\Omega}\Big\{\frac{|\nabla u|^{2}}{2}+W(u)\Big\}dx,\quad W(u)=\frac{1}{4}(1-u^{2})^{2}\chi_{[-1,1]}.
\end{equation*}
This type of energy function was originally developed by Landau, Ginzburg, and Pitaevskii \cite{GP58,Landau37,Landau67}  to describe phase transitions in thermodynamics. The minimizers of such energies satisfy the Allen-Cahn equation $$\Delta u=W'(u)=u^{3}-u.$$ A large amount of progress has been made in understanding the existence, uniqueness, symmetry, and asymptotic properties of its transition layers. For vector-valued Allen-Cahn systems, however, the situation becomes considerably more complicated due to the interaction among different components and the geometry of the zero set of the potential.

For the equation \eqref{eq. GP equation} considered in this paper, it is straightforward to verify that \eqref{eq. GP equation} is the critical point of the functional
\begin{equation}\label{eq. GL energy for BE system}
    J(\mathbf{u},\Omega)=\int_{\Omega}\Big\{\frac{1}{p}\sum_{i=1}^{m}|\nabla u^{i}|^{p}+W(\mathbf{u})\Big\}dx.
\end{equation}
In view of assumptions (A1)-(A5), it is clear that the constant states $\mathbf{u}\equiv\pm\mathbf{e}$ are two global minimizers of the energy functional  \eqref{eq. GL energy for BE system}. However, more interesting solutions are the phase transition solutions, which interpolate between values close to $-\mathbf{e}$ and $\mathbf{e}$ within a relatively narrow transition layer, commonly referred to as the phase field region.

\subsection{Existence and basic properties of the heteroclinical solution}
In this paper, we study one-dimensional phase-transitions solutions to \eqref{eq. GP equation}, especially the heteroclinical solutions defined below.

\begin{definition}\label{def. heteroclinical} A function
    $\mathbf{u}(x):\mathbb{R}\to[-\Lambda_{1},\Lambda_{1}]\times\cdots\times[-\Lambda_{m},\Lambda_{m}]$ 
       is called a heteroclinical solution to \eqref{eq. GP equation} if it satisfies the following conditions:
    \begin{itemize}
        \item[(1)]  $\mathbf{u}$ solves \eqref{eq. GP equation} in the classical sense in $\mathbb{R}$;
        \item[(2)] $\displaystyle\lim_{x\to\pm\infty}\mathbf{u}(x)=\pm\mathbf{e}$;
        \item[(3)] The set $\{x\in\mathbb{R}:\mathbf{u}(x)\neq\pm\mathbf{e}\}$ is connected,  namely, it is an interval.
    \end{itemize}
\end{definition}

The following theorem proves the existence of heteroclinical solutions and characterizes their basic qualitative behavior, including precise asymptotic estimates, monotonicity properties, and decay rates. Similar results have been obtained in \cite{AFN21,AftalionSourdis2019} for the case  $p=q_{\pm}=2$. The basic idea to prove the existence is to formulate the problem as an arc-length minimizing problem for curves in $[-\Lambda_{1},\Lambda_{1}]\times\cdots\times[-\Lambda_{m},\Lambda_{m}]$ endowed with a degenerate Riemannian/Finsler metric.
\begin{theorem}\label{thm. basic}
    Assume that (A1)-(A5) hold for the potential function $W(\mathbf{u})$. Then, there exists a heteroclinical solution to \eqref{eq. GP equation}. Moreover, this solution satisfies the identity
    \begin{equation}\label{eq. speed identity}
        |\mathbf{u}'(x)|_{l_{p}}=\sqrt[p]{\frac{p}{p-1}\cdot W(\mathbf{u}(x))},\quad\mbox{for all } x\in\mathbb{R}.
    \end{equation}
    Let $(a,b)=\{x\in\mathbb{R}^{n}:\mathbf{u}(x)\neq\pm\mathbf{e}\}$, where $a,b\in[-\infty,+\infty]$. 
    Then, for every $x\in(a,b)$,
    \[
    u_x^i(x)>0,\qquad i=1,\dots,m.
    \]
    Furthermore, as $x\to a+$, we have the following estimates for $\mathbf{u}$:
 \begin{itemize}
        \item \textbf{Singular potential well at $-\mathbf{e}$.}
        If $q_{-}<p$, then $a\neq-\infty$. Moreover,
        \[
        u^{i}(x)+\Lambda_i\sim (x-a)^{\alpha_-},\qquad i=1,\dots,m,
        \]
        where $\alpha_-=\frac{p}{p-q_-}$.
        \item \textbf{Regular potential well at $-\mathbf{e}$.}
        If $q_-=p$, then $a=-\infty$. Moreover,
        \[
        u^{i}(x)+\Lambda_i\sim \theta_-^{x},\qquad i=1,\dots,m,
        \]
        where $\theta_->1$ is a constant.
        \item \textbf{Degenerate potential well at $-\mathbf{e}$.}
        If $q_->p$, then $a=-\infty$. Moreover,
        \[
        u^{i}(x)+\Lambda_i\sim (-x)^{\alpha_-},\qquad i=1,\dots,m,
        \]
        where $\alpha_-=\frac{p}{p-q_-}$.
    \end{itemize}
    A similar asymptotic estimate holds as $x\to b-$.
\end{theorem}

The proof of Theorem~\ref{thm. basic} involves several substantial difficulties, mainly concerning the monotonicity and asymptotic behavior of heteroclinical solutions. The key obstruction arises from the lack of smoothness of the potential near its wells. Consequently, the classical dynamical system approach developed for smooth potentials, such as the one in \cite{AFN21}, cannot be directly applied to characterize the asymptotic behavior. Furthermore, when the homogeneity degree of the potential well is smaller than the diffusion exponent, namely $q_{\pm}<p$, the heteroclinical solution may reach the equilibrium states $\pm\mathbf{e}$ at finite points. In this case, the decay towards the potential wells exhibits an Alt-Phillips type free boundary behavior, as described in \cite{AP86}, which introduces additional difficulties in the analysis.

To overcome these difficulties, we develop a unified blow-up/down analysis based on three different types of Weiss-type monotonicity formulas, distinguished by the relation between the homogeneity exponents $q_{\pm}$ of the potential wells and the diffusion exponent $p$. These monotonicity formulas enable us to characterize all possible blow-up or blow-down profiles at the endpoints, covering both finite free boundary points and infinite asymptotic ends. Combining this classification with a delicate sliding argument, we establish the strict monotonicity of heteroclinical solutions.

We should specify that the Weiss-typed functions for regular and degenerate potential wells are new, as the ``free boundary points" are located at $\pm\infty$. The Weiss function for degenerate potential wells is a simple generalization of the singular one, while the one for regular potential wells takes a sightly different form, because the base number for the exponential decay rate is usually unclear.

Another essential difficulty arises in proving the monotonicity of each component, especially when applying the sliding method. The main obstruction is the loss of ellipticity of the linearized operator associated with the $p$-Laplacian when the gradient vanishes for $p\neq 2$. As a consequence, the standard maximum principle and strong maximum principle arguments cannot be directly applied. To overcome this difficulty, we first establish the monotonicity of $\mathbf{u}$ in the far field by means of the Weiss monotonicity formula (see also Corollary~\ref{cor3.1}). Then, using the speed identity \eqref{eq. speed identity}, we recover uniform ellipticity near the possible touching point arising in the sliding method, which allows us to complete the monotonicity argument.

\subsection{Motivation for the curvature estimate}
In \cite{S09}, by deforming the heteroclinical solution of \eqref{eq. GP equation} with $m=1$ and $p=2$, namely
\begin{equation}\label{eq. Allen-Cahn}
    \Delta u=W'(u)\quad\mbox{where }u(x)\in[-1,1],
\end{equation}
a family of test and barrier functions was constructed to investigate the curvature of the level sets of solutions to the Allen-Cahn equation. This construction led to an ABP estimate for the level sets. Based on this estimate, Savin applied the ``improvement of flatness'' method to classify asymptotically flat global solutions of \eqref{eq. Allen-Cahn}, providing a partial affirmative answer to the De Giorgi conjecture. See also \cite{SV05,VSS06} for analogous results concerning the $p$-Laplacian Allen-Cahn equation. More recently, another barrier function was introduced in \cite{SZ26} to establish density estimates for degenerate Allen-Cahn equations.

Motivated by these insights, we expect that a similar construction of test or barrier functions can be carried out in the Bose-Einstein setting. In particular, our goal is to construct a radially symmetric super-solution to \eqref{eq. GP equation} with $m\geq2$.

The main idea is to construct a vector-valued function $\mathbf{g}(t)$ near $t=0$ such that, for some sufficiently large $R$,
\begin{equation}\label{eq. radial barrier supersolution}
    \frac{d}{dt}\Big(|g^{i}_{t}|^{p-2}\cdot g^{i}_{t}\Big)+\frac{n-1}{R+t}\cdot|g^{i}_{t}|^{p-2}\cdot g^{i}_{t}\leq \nabla_{i}W(\mathbf{g}).
\end{equation}
Then, we consider the annulus
\begin{equation*}
    A_{R,r}=\{X\in\mathbb{R}^{n}:R-r\leq|X+R\vec{e}_{n}|\leq R+r\}.
\end{equation*}
Let $\phi(X):A_{R,r}\to[-\Lambda_{1},\Lambda_{1}]\times\cdots\times[-\Lambda_{m},\Lambda_{m}]$ be the following function for $X\in A_{R,r}$:
\begin{equation}\label{eq. rotate the barrier function}
    \phi(X)=\mathbf{g}(|X+R\vec{e}_{n}|-R).
\end{equation}
If $\mathbf{g}$ satisfies \eqref{eq. radial barrier supersolution} for $|t|\leq r$, then
\begin{equation*}
    \Delta_{p}\phi^{i}(X)\leq\nabla_{i}W(\phi(X)),\quad\mbox{for all }X\in A_{R,r}.
\end{equation*}

Throughout this paper, we use the notation $|\cdot|_{l_p}$ to denote the $l_p$-norm on $\mathbb{R}^m$:
\begin{equation*}
    |v|_{l_{p}}=\Big(\sum_{i=1}^{m}|v^{i}|^{p}\Big)^{1/p},\quad\mbox{where }v=(v^{1},\cdots,v^{m}).
\end{equation*}
The trajectory of a heteroclinical solution $\mathbf{u}$ is defined by
\[
\gamma(s)=\mathbf{u}(x(s)),
\]
where the parameter $s$ is chosen such that
\[
|\gamma_s|_{l_p}\equiv1.
\]
Thus, $\gamma$ is an $l_p$-arc-length re-parametrization of the trajectory of $\mathbf{u}$.

The function $\mathbf{g}(t)$ is expected to be a suitable re-parametrization of a heteroclinical solution $\mathbf{u}$ of \eqref{eq. GP equation}. In practice, $\mathbf{g}(t)$ is constructed as a deformation of the trajectory $\gamma(s)$ introduced above. More details are provided in Section~\ref{sec. casual talk}. 

From \eqref{eq. why curvature is important}, namely,
\begin{equation*}
    \mathcal{E}'(s)\cdot\gamma_s^i
    +p\,\mathcal{E}(s)\cdot\gamma_{ss}^i>0,
\end{equation*}
we observe that the curvature of the trajectory, represented by $\gamma_{ss}$, constitutes the main error term. Such a curvature term does not appear in the construction of test or barrier functions for the Allen-Cahn equation \eqref{eq. Allen-Cahn}, since the range of the solution, namely the interval $[-1,1]$, is one-dimensional and has zero curvature. However, in the Bose-Einstein setting with $m\geq2$, the trajectory of $\mathbf{u}$ generally has nonzero curvature, making the verification of \eqref{eq. why curvature is important} considerably more involved.

Roughly speaking, we need to establish the following two estimates:
\begin{itemize}
    \item $\gamma^{i}_{s}$ is comparable with $\gamma^{j}_{s}$ for $0\leq i,j\leq m$,
    \item $\gamma^{i}_{ss}$ can somehow be controlled by $\gamma^{i}_{s}$.
\end{itemize}
The precise meaning of ``is comparable with" and ``can somehow be controlled by" is presented in the following main result.
\begin{theorem}\label{thm. main}
    Assume that (A1)-(A5) hold for the potential function $W(\mathbf{u})$, and let $\mathbf{u}(x)$ be the heteroclinical solution to \eqref{eq. GP equation} constructed in Theorem~\ref{thm. basic}. Let $\gamma(s)$ denote the $l_{p}$-arc-length re-parametrization of $\mathbf{u}(x)$, such that
    \begin{equation*}
        \gamma(s)=\mathbf{u}(x(s)),\quad\mbox{where }\frac{dx}{ds}=\Big(|\mathbf{u}_{x}|_{l_{p}}\Big)^{-1}=\Big(\frac{p}{p-1}\cdot W(\mathbf{u})\Big)^{-1/p}.
    \end{equation*}
       After a translation, we may assume that $\gamma(s)$ is defined on $[s_-,s_+]$ and satisfies
    \begin{equation*}
        \gamma(s)\neq\pm\mathbf{e}\mbox{ for }s\in(s_{-},s_{+}),\quad\mbox{and }\lim_{s\to s_{\pm}}\gamma(s)=\pm\mathbf{e}.
    \end{equation*}
    Then the following estimates hold:
    \begin{itemize}
        \item[(1)] There exists a constant $0<C<\infty$, such that
        \begin{equation*}
            0<\gamma^{i}_{s}(s)\leq C\gamma^{j}_{s}(s),\quad\mbox{for all }s\in(s_{-},s_{+})\mbox{ and }1\leq i,j\leq m;
        \end{equation*}
        \item[(2)] As $s\to s_{\pm}$, one has
        \begin{equation*}
            \lim_{s\to s_{-}}\frac{|s-s_{-}|\cdot|\gamma_{ss}(s)|}{|\gamma_{s}(s)|}=0,\quad\mbox{and }\lim_{s\to s_{+}}\frac{|s-s_{+}|\cdot|\gamma_{ss}(s)|}{|\gamma_{s}(s)|}=0.
        \end{equation*}
    \end{itemize}
\end{theorem}

Part (1) of Theorem~\ref{thm. main} is a direct consequence of the monotonicity result established in Lemma \ref{le5.4} together with Corollary~\ref{cor3.1}. In particular, the first estimate implies that the components of the tangent vector of the heteroclinic trajectory remain uniformly comparable near the potential wells, which reflects the ``multiplicity one" structure of the interface. Part (2) follows from \eqref{eq. how to calculate the magnetude of the curvature}, which provides an indirect characterization of the curvature of the trajectory. The resulting estimate shows that the curvature becomes negligible compared with the tangent vector at the scale of the distance to the endpoint, indicating the asymptotic flatness of the heteroclinic trajectory near the wells. This asymptotic flatness property is expected to play an essential role in the construction of barrier functions and in the analysis of higher-dimensional interfaces.

The paper is organized as follows:
\begin{itemize}
    \item In Section~2, we prove the existence of heteroclinical solutions as well as the speed identity.
    \item In Section~3, we establish the fundamental qualitative properties, especially basic decay estimates and asymptotic homogeneity near the endpoints.
    \item Section~4 is devoted to the asymptotic analysis of heteroclinical solutions near the potential wells. We investigate all types of potential wells and characterize the corresponding asymptotic behavior.
    \item Section~5 is devoted to the monotonicity in each component for the heteroclinical solution.
    \item In Section~6, we study the geometric properties of the heteroclinic trajectories. By introducing an appropriate $l_p$-arc-length parametrization, we derive estimates for the tangent vectors and curvature of the trajectories, which describe their asymptotic flatness near the potential wells.
    \item Section~7 is devoted to further analysis and applications of these estimates, including the construction of suitable barrier functions and related consequences.
\end{itemize}

\section{Existence}
In this section, we prove the existence part of Theorem~\ref{thm. basic} and establish the identity \eqref{eq. speed identity}. We first state the existence result.

\begin{lemma}[Existence]\label{lem. existence}
    Assume that (A1)-(A5) hold. Then, there exists at least one heteroclinical solution $\mathbf{u}$ to \eqref{eq. GP equation}. Moreover, \eqref{eq. speed identity} holds for $\mathbf{u}$.
\end{lemma}
\begin{proof}
    \textbf{Step 1: Construction of a minimizing sequence of curves.} 
    We introduce the following class of admissible curves:
    \begin{equation*}
        \mathcal{C}:=\Big\{\gamma\in C^{\infty}([-1,1]):\gamma^{-1}(\{\pm\mathbf{e}\})=\{\pm1\},\ \inf_{t\in[-1,1]}|\gamma'(t)|_{l_{p}}>0\Big\}.
    \end{equation*}
     For $\gamma\in\mathcal{C}$, we consider the non-negative functional
    \begin{equation*}
        \mathcal{F}(\gamma)=\int_{-1}^{1}\Big(\frac{p}{p-1}\cdot W(\gamma(t))\Big)^{\frac{p-1}{p}}\cdot|\gamma'(t)|_{l_{p}}dt.
    \end{equation*}
    It is straightforward to verify that $\mathcal{F}(t\cdot\mathbf{e})<\infty$ for $t\cdot\mathbf{e}\in\mathcal{C}$.  Hence, there exists a minimizing
    sequence $\{\gamma_j\}_{j=1}^{\infty}\subset\mathcal{C}$ such that
    \begin{equation*}
        \lim_{j\to\infty}\mathcal{F}(\gamma_{j})=\sigma:=\inf_{\gamma\in\mathcal{C}}\mathcal{F}(\gamma).
    \end{equation*}
    By a suitable re-parametrization (still $\gamma_{j}\in\mathcal{C}$), we may assume that:
    \begin{equation*}
        |\gamma_{j}'(t)|_{l_{p}}\equiv\mu_{j}\mbox{ for all }t\in[-1,1].
    \end{equation*}
     Next, for each $j$, let $t_j\in(-1,1)$ be the unique point satisfying
    \begin{equation*}
        \int_{-1}^{t_{j}}\Big(\frac{p}{p-1}\cdot W(\gamma_{j}(t))\Big)^{\frac{p-1}{p}}\cdot|\gamma_{j}'(t)|_{l_{p}}dt=\int_{t_{j}}^{1}\Big(\frac{p}{p-1}\cdot W(\gamma_{j}(t))\Big)^{\frac{p-1}{p}}\cdot|\gamma_{j}'(t)|_{l_{p}}dt=\frac{1}{2}\mathcal{F}(\gamma_{j}).
    \end{equation*}
    
    \textbf{Step 2: Re-parametrization.} 
    For each $j$, we define a re-parametrized curve
\[
\mathbf{u}_j(x)=\gamma_j(t),
\]
where $x=x(t)$ is determined by the following ODE problem on $\mathbb{R}$:
    \begin{equation*}
        \left\{\begin{aligned}
            &x'(t)=\mu_{j}\cdot\Big(\frac{p}{p-1}\cdot W(\gamma_{j}(t))\Big)^{-\frac{1}{p}},\\
            &x(t_{j})=0.
        \end{aligned}\right.
    \end{equation*}
    By the chain rule and the definition of the re-parametrization, we obtain
the identity  for all $j$:
    \begin{equation*}
        |\nabla\mathbf{u}_{j}(x)|_{l_{p}}=|\mathbf{u}_{j}'(x)|_{l_{p}}=\Big(\frac{p}{p-1}\cdot W(\mathbf{u}_{j}(x))\Big)^{\frac{1}{p}},\quad\mbox{i.e.,}\sup_{j}\sup_{x\in\mathbb{R}}|\nabla\mathbf{u}_{j}(x)|_{l_{p}}<\infty.
    \end{equation*}
    This also implies the following important identity:
    \begin{equation}\label{eq. speed identity for the sequence}
        \frac{1}{p}\sum_{i=1}^{m}|\frac{d}{dx}u_{j}^{i}|^{p}+W(\mathbf{u}_{j}(x))=\Big(\frac{p}{p-1}\cdot W(\mathbf{u}_{j}(x))\Big)^{\frac{p-1}{p}}\cdot|\mathbf{u}_{j}'(x)|_{l_{p}}.
    \end{equation}
   Finally, another application of the chain rule gives 
    \begin{equation*}
        J(\mathbf{u}_{j},\mathbb{R})=\mathcal{G}(\mathbf{u}_{j})=\mathcal{F}(\gamma_{j}),
    \end{equation*}
    where
    \begin{equation*}
        \mathcal{G}(\mathbf{u}_{j}):=\int_{-\infty}^{+\infty}\Big(\frac{p}{p-1}\cdot W(\mathbf{u}_{j}(x))\Big)^{\frac{p-1}{p}}\cdot|\mathbf{u}_{j}'(x)|_{l_{p}}dx.
    \end{equation*}

    \textbf{Step 3: Passing to the limit.} Since 
    \[
    \sup_j\sup_{x\in\mathbb{R}}
    |\nabla\mathbf{u}_j(x)|_{l_p}<\infty,
\]
the sequence $\{\mathbf{u}_j\}$ is uniformly Lipschitz continuous on
every compact subset of $\mathbb{R}$. Therefore, by the
Arzel\`a-Ascoli theorem and a diagonal argument, after passing to a
subsequence (still denoted by $\mathbf{u}_j$), there exists a function
$\mathbf{u}\in C^{0,1}(\mathbb{R})$ such that
    \begin{itemize}
        \item $\mathbf{u}_{j}\to\mathbf{u}$ in $C^{\alpha}_{loc}(\mathbb{R})$ for all $\alpha\in(0,1)$;
        \item $\mathbf{u}_{j}\to\mathbf{u}$ in $W^{1,\beta}_{loc}(\mathbb{R})$ for all $\beta\in(1,+\infty)$;
        \item $\mathbf{u}_{j}\rightharpoonup\mathbf{u}$ in $W^{1,p}(\mathbb{R})$.
    \end{itemize}
    The convergence in
$C^\alpha_{\mathrm{loc}}(\mathbb{R})\cap W^{1,\beta}_{\mathrm{loc}}(\mathbb{R})$,
together with \eqref{eq. speed identity for the sequence}, implies that,
for every fixed $L>0$,
    \begin{align*}
        \int_{-L}^{L}\Big\{\frac{1}{p}\sum_{i=1}^{m}|\frac{d}{dx}u_{j}^{i}|^{p}+W(\mathbf{u}_{j}(x))\Big\}dx\to&\int_{-L}^{L}\Big\{\frac{1}{p}\sum_{i=1}^{m}|\frac{d}{dx}u^{i}|^{p}+W(\mathbf{u}(x))\Big\}dx,\\
        \int_{-L}^{L}\Big(\frac{p}{p-1}\cdot W(\mathbf{u}_{j}(x))\Big)^{\frac{p-1}{p}}\cdot|\mathbf{u}_{j}'(x)|_{l_{p}}dx\to&\int_{-L}^{L}\Big(\frac{p}{p-1}\cdot W(\mathbf{u}(x))\Big)^{\frac{p-1}{p}}\cdot|\mathbf{u}'(x)|_{l_{p}}dx.
    \end{align*}
    In particular, for all $L>0$, we have
    \begin{equation*}
        \int_{-L}^{L}\Big\{\frac{1}{p}\sum_{i=1}^{m}|\frac{d}{dx}u^{i}|^{p}+W(\mathbf{u}(x))\Big\}dx=\int_{-L}^{L}\Big(\frac{p}{p-1}\cdot W(\mathbf{u}(x))\Big)^{\frac{p-1}{p}}\cdot|\mathbf{u}'(x)|_{l_{p}}dx,
    \end{equation*}
    which implies that (using the equality condition of the Young's inequality):
    \begin{equation}\label{eq. speed identity almost everywhere}
        |\mathbf{u}'(x)|_{l_{p}}=\sqrt[p]{\frac{p}{p-1}\cdot W(\mathbf{u}(x))},\quad\mbox{for almost every } x\in\mathbb{R}.
    \end{equation}
   Equivalently, \eqref{eq. speed identity} holds almost everywhere.

    \textbf{Step 4: $\mathbf{u}(0)\neq\pm\mathbf{e}$.} It suffices to prove the following stronger estimate:
    \begin{equation}\label{eq. u_j's are all balanced}
        \inf_{j}\inf_{x\geq0}|\mathbf{u}_{j}(x)+\mathbf{e}|>0,\quad\mbox{and }\inf_{j}\inf_{x\leq0}|\mathbf{u}_{j}(x)-\mathbf{e}|>0.
    \end{equation}
    Suppose by contradiction that the first inequality of \eqref{eq. u_j's are all balanced} fails, then $u_{j^{*}}(x^{*})$ is sufficiently close to $-\mathbf{e}$ for some $(j^{*},x^{*})\in\mathbb{Z}_{+}\times\mathbb{R}_{+}$. Let $t^{*}\in[-1,1]$ satisfy 
    $$\gamma_{j^{*}}(t^{*})=\mathbf{u}_{j^{*}}(x^{*}).$$ Since $x^{*}\geq0$, we must have $t^{*}\geq t_{j^{*}}$. Then,
    \begin{equation*}
        \int_{t^{*}}^{1}\Big(\frac{p}{p-1}\cdot W(\gamma_{j^{*}}(t))\Big)^{\frac{p-1}{p}}\cdot|\gamma_{j^{*}}'(t)|_{l_{p}}dt\leq\frac{\mathcal{F}(\gamma_{j^{*}})}{2}\leq\frac{\sigma}{2}+o(1).
    \end{equation*}
    Using this estimate, we can construct a curve $\widetilde{\gamma}_{j^*}\in\mathcal{C}$ such that
    \begin{itemize}
        \item $\widetilde{\gamma}_{j^{*}}(t)\equiv\gamma_{j^{*}}(t)$ for $t\in[t^{*},1]$;
        \item $\displaystyle\int_{-1}^{t^{*}}\Big(\frac{p}{p-1}\cdot W(\widetilde{\gamma}_{j^{*}}(t))\Big)^{\frac{p-1}{p}}\cdot|\widetilde{\gamma}_{j^{*}}'(t)|_{l_{p}}dt$ is sufficiently small.
    \end{itemize}
    This violates the minimality of $\sigma$, as $\displaystyle\sigma:=\inf_{\gamma\in\mathcal{C}}\mathcal{F}(\gamma)\leq\frac{\sigma}{2}+o(1)$. Then, \eqref{eq. u_j's are all balanced} is verified.

    \textbf{Step 5: $\mathbf{u}$ is energy minimizing.} Suppose on the contrary that $$\mathbf{v}:\mathbb{R}\to[-\Lambda_{1},\Lambda_{1}]\times\cdots\times[-\Lambda_{m},\Lambda_{m}]$$ is a $C^{0,1}(\mathbb{R})$ function satisfying:
    \begin{itemize}
        \item $\mathbf{v}(x)=\mathbf{u}(x)$ for all $|x|\geq L$ for some $L\geq100$;
        \item $J(\mathbf{v},[-L,L])\leq J(\mathbf{u},[-L,L])-\delta$ for some $\delta>0$.
    \end{itemize}
    Applying Young's inequality to $\mathbf{v}$ and using \eqref{eq. speed identity almost everywhere}, we obtain
    \begin{equation*}
        \int_{-L}^{L}\Big(\frac{p}{p-1}\cdot W(\mathbf{v}(x))\Big)^{\frac{p-1}{p}}\cdot|\mathbf{v}'(x)|_{l_{p}}dx\leq\int_{-L}^{L}\Big(\frac{p}{p-1}\cdot W(\mathbf{u}(x))\Big)^{\frac{p-1}{p}}\cdot|\mathbf{u}'(x)|_{l_{p}}dx-\delta.
    \end{equation*}
    Then, we can replace $\mathbf{u}_{j}$ with a competitor $\mathbf{w}_{j}$, such that:
    \begin{itemize}
        \item $\mathbf{w}_{j}(x)\equiv\mathbf{u}_{j}(x)$ for all $|x|\geq L$;
        \item $\mathbf{w}_{j}(x)=\mathbf{v}(\frac{L}{L-1}x)$ for all $|x|\leq L-1$;
        \item $\mathbf{w}_{j}(x)$ is affine for $x\in[L-1,L]$ and for $x\in[-L,1-L]$.
    \end{itemize}

Since
\[
\mathbf u_j\to\mathbf u
\quad\mbox{in }
C^\alpha_{\mathrm{loc}}(\mathbb R)
\cap
W^{1,\beta}_{\mathrm{loc}}(\mathbb R),
\]
for sufficiently large $j$ we obtain
\begin{equation*}
    \mathcal{G}(\mathbf w_j)
    \leq
    \mathcal{G}(\mathbf u_j)-\frac{\delta}{2}
    \leq
    \sigma-\frac{\delta}{4}.
\end{equation*}
After re-parametrizing $\mathbf w_j$, we obtain a sequence
$\widetilde{\gamma}_j\in\mathcal C$ satisfying
\[
    \mathcal{F}(\widetilde{\gamma}_j)
    \leq
    \sigma-\frac{\delta}{4},
\]
which contradicts the definition of $\sigma$.

Therefore, $\mathbf u$ is energy minimizing on all bounded intervals.
In particular, $\mathbf u$ satisfies the Euler-Lagrange equation and is
therefore a weak solution (and hence a classical solution) of
\eqref{eq. GP equation} in $\mathbb R$. As a consequence,  the ``almost everywhere identity" \eqref{eq. speed identity almost everywhere} is improved and it becomes \eqref{eq. speed identity}.

    \textbf{Step 6: $\mathbf{u}(\pm\infty)=\pm\mathbf{e}$.} 
    By the weak convergence of $\mathbf{u}_j$ in $W^{1,p}(\mathbb{R})$,
together with \eqref{eq. speed identity for the sequence} and
\eqref{eq. speed identity almost everywhere}, we obtain the following
energy estimate:
    \begin{equation*}
        J(\mathbf{u},\mathbb{R})=\int_{-\infty}^{\infty}\sum_{i=1}^{m}|\frac{d}{dx}u^{i}|^{p}dx\leq\liminf_{j\to\infty}\int_{-\infty}^{\infty}\sum_{i=1}^{m}|\frac{d}{dx}u_{j}^{i}|^{p}dx=\liminf_{j\to\infty}J(\mathbf{u}_{j},\mathbb{R})=\sigma.
    \end{equation*}
    
    From the $C^{1,\epsilon}$ estimate, we know that $|\mathbf{u}'(x)|_{l_{p}}$ is globally bounded from above. Then, for all $x\in\mathbb{R}$ with $\mathbf{u}(x)\in\Big(B_{\epsilon}(-\mathbf{e})\cup B_{\epsilon}(\mathbf{e})\Big)^{c}$ ($\epsilon>0$ is arbitrary), we have 
    \[
W(\mathbf{u}(y))\geq c(\epsilon)
\quad\mbox{for }y\in[x-c(\epsilon),x+c(\epsilon)],
\]
 where $c(\epsilon)>0$ is independent of $x$. Therefore, using the fact that $J(\mathbf{u},\mathbb{R})\leq\sigma<\infty$, we see that for each $\epsilon>0$, the preimage of $\Big(B_{\epsilon}(-\mathbf{e})\cup B_{\epsilon}(\mathbf{e})\Big)^{c}$ under $\mathbf{u}$ must be bounded. Using the $C^{1,\epsilon}$ estimate again (to avoid a sudden change of $\mathbf{u}$ between $\pm\mathbf{e}$), we have that
    \begin{equation*}
        \lim_{x\to-\infty}\mathbf{u}(x)\in\{-\mathbf{e},\mathbf{e}\},\quad\mbox{and }\lim_{x\to+\infty}\mathbf{u}(x)\in\{-\mathbf{e},\mathbf{e}\}.
    \end{equation*}
    With the help of \eqref{eq. u_j's are all balanced}, we see that the only possibility is that $\mathbf{u}(\pm\infty)=\pm\mathbf{e}$.

    \textbf{Step 7: $\{\mathbf{u}\neq\pm\mathbf{e}\}$ is connected.} We already know that $\mathbf{u}(\pm\infty)=\pm\mathbf{e}$, then there exists at least one interval $(a,b)$ with $a,b\in[-\infty,+\infty]$, such that:
    \begin{equation*}
        \mathbf{u}(a+)=-\mathbf{e},\quad\mathbf{u}(b-)=\mathbf{e},\quad\mbox{and }\mathbf{u}(x)\neq\pm\mathbf{e}\mbox{ in }(a,b).
    \end{equation*}
    By Young's inequality and by the minimality of $\sigma$, we must have
    \begin{equation*}
        J(\mathbf{u},(a,b))\geq\int_{a}^{b}\Big(\frac{p}{p-1}\cdot W(\mathbf{u}(x))\Big)^{\frac{p-1}{p}}\cdot|\mathbf{u}'(x)|_{l_{p}}dx\geq\sigma.
    \end{equation*}
    As $J(\mathbf{u},\mathbb{R})\leq\sigma$, we see that $\mathbf{u}\in\{-\mathbf{e},\mathbf{e}\}$ in $\mathbb{R}\setminus(a,b)$, so $\{\mathbf{u}\neq\pm\mathbf{e}\}$ is connected.

    \textbf{Ending.} We have constructed the function $\mathbf{u}$ and verified all the
required properties. This completes the proof of
Lemma~\ref{lem. existence}.
\end{proof}

\section{Decay rate estimate}
\subsection{Existence and non-existence of free boundary}
In this section, we prove the decay estimate in Theorem~\ref{thm. basic} under assumptions (A1)-(A5). Our analysis relies on the energy minimizing property and the speed identity of the heteroclinical solution constructed in Lemma~\ref{lem. existence}.

We first establish an important lower bound estimate as $\mathbf{u}$ approaches $\pm\mathbf{e}$.
\begin{lemma}\label{lem. lower bound of decay}
    Let $\mathbf{u}$ be the solution constructed in Lemma~\ref{lem. existence}. Assume that 
    \[
        |\mathbf{u}(x)+\mathbf{e}|=h
    \]
    for some sufficiently small $h>0$. Then there exists a small constant
    $c_{1}=c_{1}(p,q_{-})>0$ such that
    \begin{equation*}
        |\mathbf{u}(y)+\mathbf{e}|\geq c_{1}h,\quad\mbox{for every }x-c_{1}h^{\frac{p-q_{-}}{p}}\leq y\leq x+c_{1}h^{\frac{p-q_{-}}{p}}.
    \end{equation*}
\end{lemma}
\begin{proof}
 By translation invariance, we may assume that $x=0$. Define the auxiliary function
    \begin{equation*}
        D(y)=\left\{\begin{aligned}
            &\frac{p}{p-q_{-}}\Big(\sum_{i=1}^{m}(u^{i}(y)+\Lambda_{i})\Big)^{\frac{p-q_{-}}{p}},&\mbox{if }&q_{-}\neq p,\\
            &\ln{\Big(\sum_{i=1}^{m}(u^{i}(y)+\Lambda_{i})\Big)},&\mbox{if }&p=q_{-}.
        \end{aligned}\right.
    \end{equation*}
   According to the value of $q_{-}$, the function $D$ takes values in
    \begin{equation*}
        D(y)\in\left\{\begin{aligned}
            &(0,+\infty),&\mbox{if }&q_{-}<p,\\
            &[-\infty,+\infty),&\mbox{if }&q_{-}=p,\\
            &[-\infty,0),&\mbox{if }&q_{-}>p.
        \end{aligned}\right.
    \end{equation*}
   Differentiating $D$  and employing the identity\eqref{eq. speed identity}, we obtain
    \begin{equation*}
        |D'(y)|=\Big(\sum_{i=1}^{m}(u^{i}(y)+\Lambda_{i})\Big)^{\frac{-q_{-}}{p}}\cdot\Big|\sum_{i=1}^{m}u^{i}_{x}\Big|\lesssim W(\mathbf{u})^{-\frac{1}{p}}\cdot|\mathbf{u}_{x}|_{l_{p}}\sim1.
    \end{equation*}
    Since $|u(0)+\mathbf{e}|=h$, we have
    \begin{equation*}
        D(0)\sim\left\{\begin{aligned}
            &\frac{p}{p-q_{-}}(c\cdot h)^{\frac{p-q_{-}}{p}},&\mbox{if }&q_{-}\neq p,\\
            &\ln{(c\cdot h)},&\mbox{if }&p=q_{-}.
        \end{aligned}\right.
    \end{equation*}
    Then, if $-c_{1}h^{\frac{p-q_{-}}{p}}\leq y\leq c_{1}h^{\frac{p-q_{-}}{p}}$, we use the boundedness of $|D'(y)|$ and get that
    \begin{equation*}
        D(y)\sim\left\{\begin{aligned}
            &\frac{p}{p-q_{-}}(c\cdot h)^{\frac{p-q_{-}}{p}},&\mbox{if }&q_{-}\neq p,\\
            &\ln{(c\cdot h)},&\mbox{if }&p=q_{-}.
        \end{aligned}\right.
    \end{equation*}
    In other words, $|\mathbf{u}(y)+\mathbf{e}|\geq c_{1}h$.
\end{proof}

The following lemma provides an upper bound estimate for the decay of $\mathbf{u}$ as it approaches $\pm\mathbf{e}$ (up to a subsequence).

\begin{lemma}\label{lem. upper bound (subsequence)}
    Let $\mathbf{u}$ be the solution constructed in Lemma~\ref{lem. existence}. Assume that $|\mathbf{u}(x)+\mathbf{e}|=h$ for some sufficiently small $h>0$. Then there exists a sufficiently large constant $C_{2}=C_{2}(p,q_{-})$ such that
    \begin{equation*}
        |\mathbf{u}(y)+\mathbf{e}|\leq c_{1}h,\quad\mbox{for some }x-C_{2}h^{\frac{p-q_{-}}{p}}\leq y\leq x.
    \end{equation*}
    Here, the constant $c_{1}$ is the same as the one in Lemma~\ref{lem. upper bound (subsequence)}.
\end{lemma}
\begin{proof}
    For simplicity, we assume that $x=0$. We first derive an upper bound for the energy $J(\mathbf{u},\mathbb{R}_{-})$. To this end, we consider the following competitor in $\mathbb{R}_{-}$:
    \begin{equation*}
        \mathbf{v}(y)=\max\{1+h^{\frac{q_{-}-p}{p}}\cdot y,0\}\cdot \mathbf{u}(0).
    \end{equation*}
    It is easy to verify that $\mathbf{v}(0-)=\mathbf{u}(0-)$ and $\mathbf{v}(-\infty)=\mathbf{u}(-\infty)$. Therefore, using the energy minimizing property of $\mathbf{u}$ (see \textbf{Step 5} in the proof of Lemma~\ref{lem. existence}), we must have
    \begin{equation*}
        J(\mathbf{u},\mathbb{R}_{-})\leq J(\mathbf{v},\mathbb{R}_{-})\lesssim h^{\frac{p-q_{-}}{p}}\cdot h^{q_{-}}.
    \end{equation*}
    If the conclusion of Lemma~\ref{lem. upper bound (subsequence)} does not hold for some large $C_{2}$, then $W(\mathbf{u})\gtrsim(c_{1}h)^{q}$ for $-C_{2}h^{\frac{p-q_{-}}{q}}\leq y\leq0$. As a result, we have
    \begin{equation*}
        h^{\frac{p-q_{-}}{p}}\cdot h^{q_{-}}\gtrsim J(\mathbf{u},\mathbb{R}_{-})\geq\int_{-C_{2}h^{\frac{p-q_{-}}{q}}}^{0}W(\mathbf{u}) dx\gtrsim C_{2}h^{\frac{p-q_{-}}{p}}\cdot(c_{1}h)^{q_{-}}.
    \end{equation*}
    Then, we have reached a contradiction when $C_{2}$ is large.
\end{proof}

Summarizing Lemma~\ref{lem. lower bound of decay} and Lemma~\ref{lem. upper bound (subsequence)}, we have the following estimate.
\begin{lemma}\label{lem. (non-)existence of free boundary}
    Let $\mathbf{u}$ be the solution constructed in Lemma~\ref{lem. existence}. Assume that at $x\in\mathbb{R}$ $|\mathbf{u}(x)+\mathbf{e}|=h$ for some sufficiently small $h>0$. Then the behavior of $\mathbf{u}$ near $-\mathbf{e}$ is classified into the following three cases:
    \begin{itemize}
        \item[(1)] If $q_{-}<p$, there exists some $a\leq x$ with $x-a\sim h^{\frac{p-q_{-}}{p}}$. For all $y\in[a,x]$, we have $|\mathbf{u}(y)+\mathbf{e}|\sim(y-a)^{\frac{p}{p-q_{-}}}$.
        \item[(2)] If $q_{-}=p$, then for all $y\leq x$, $C^{-1}\cdot(\theta_{2})^{y-x}\cdot h\leq|\mathbf{u}(y)+\mathbf{e}|\leq C\cdot(\theta_{1})^{y-x}\cdot h$ for some $0<\theta_{1}\leq \theta_{2}<1$ depending on $W(\mathbf{u})$.
        \item[(3)] If $q_{-}>p$, then for all $y\leq x$, $|\mathbf{u}(y)+\mathbf{e}|\sim(x-y+1)^{\frac{p}{p-q_{-}}}\cdot h$.
    \end{itemize}
    A similar estimate holds when $|\mathbf{u}(x)-\mathbf{e}|=h$ for a sufficiently small $h>0$.
\end{lemma}
\begin{proof}
    We only consider the case $q_{-}<p$ (which is the most complicated situation), and the other two cases can be argued similarly. By Lemma~\ref{lem. lower bound of decay} and Lemma~\ref{lem. upper bound (subsequence)}, there exists a sequence $\{x_k\}_{k\geq0}$ with $x_0=x$ such that
    \begin{equation*}
        |\mathbf{u}(x_{k})+\mathbf{e}|=(c_{1})^{k}h,\quad\mbox{and }x_{k}-C_{2}\Big((c_{1})^{k}h\Big)^{\frac{p-q_{-}}{p}}\leq x_{k+1}\leq x_{k}-c_{1}\Big((c_{1})^{k}h\Big)^{\frac{p-q_{-}}{p}}.
    \end{equation*}
    Then, such a sequence must converge, and we let $\displaystyle a:=\lim_{k\to\infty}x_{k}$. Moreover,
    \begin{equation*}
        x_{k}-a\sim\Big((c_{1})^{k}h\Big)^{\frac{p-q_{-}}{p}},
    \end{equation*}
    and, in particular,
\[
x-a\sim h^{\frac{p-q_-}{p}} .
\]
    Let $y\in[a,x]$, then either $y=a$ or $y\in[x_{k+1},x_{k}]$ for some $k\geq0$. If $y=a$, then
    \begin{equation*}
        |\mathbf{u}(y)+\mathbf{e}|=|\mathbf{u}(a)+\mathbf{e}|=\lim_{k\to\infty}|\mathbf{u}(x_{k})+\mathbf{e}|=0.
    \end{equation*}
    If $y\in[x_{k+1},x_{k}]$ for some $k\geq0$, then 
    $$y-a\sim\Big((c_{1})^{k}h\Big)^{\frac{p-q_{-}}{p}}.$$ 
    Applying Lemma~\ref{lem. lower bound of decay} with the input being replaced by $(x_{k},(c_{1})^{k}h)$, we have $$|\mathbf{u}(y)+\mathbf{e}|\geq(c_{1})^{k+1}h. $$
    Besides, if $|u(y)+\mathbf{e}|=H\gg(c_{1})^{k}h$, we can also apply Lemma~\ref{lem. lower bound of decay} with the input being replaced by $(y,H)$. As a result, we have $|\mathbf{u}(x_{k})+\mathbf{e}|\gtrsim H\gg(c_{1})^{k}h$, which is a contradiction. Therefore, we concluede that $$|u(y)+\mathbf{e}|\sim(c_{1})^{k}h\sim(y-a)^{\frac{p}{p-q_{-}}}.$$
\end{proof}
\subsection{Blow-up or blow-down}\label{subsec. Blow-up or blow-down}
We further study the asymptotic behavior near the potential wells and improve the estimates obtained in Lemma~\ref{lem. (non-)existence of free boundary}. For simplicity, we only study the asymptotic behavior as $\mathbf{u}\to-\mathbf{e}$. Define
\begin{equation*}
    \mathbf{v}(x):=\mathbf{u}(x)+\mathbf{e}.
\end{equation*}
Since $\mathbf{u}(x)\in[-\Lambda_{1},\Lambda_{1}]\times[-\Lambda_{m},\Lambda_{m}]$, we have 
\[
v^{i}(x)\geq0,\qquad 1\leq i\leq m.
\]

From Lemma~\ref{lem. (non-)existence of free boundary}, we know that $a\in\mathbb{R}$ when $q_{-}<p$, while $a=-\infty$ when $q_{-}\geq p$. Without loss of generality, we assume that $a=0$ when $q_{-}<p$. 
We now consider the limit $T\to a+$, that is,
\[
T\to 0 +\quad\text{if }q_{-}<p,
\qquad\text{and}\qquad
T\to-\infty\quad\text{if }q_{-}\geq p.
\]

\begin{definition}\label{def. blow-up/down process, v_T construction}
Consider the following blow-up/down functions.
\begin{itemize}
    \item[(1)] When $q_{-}<p$ (and $a=0$), as $T\to0+$, we let
    \begin{equation*}
        \mathbf{v}_{T}(y):=T^{-\frac{p}{p-q_{-}}}\cdot\mathbf{v}(Ty),\quad\mbox{for }y\in[0,1].
    \end{equation*}
    \item[(2)] When $q_{-}=p$ (and $a=-\infty$), as $T\to-\infty$, we let
    \begin{equation*}
        \mathbf{v}_{T}(y):=|\mathbf{v}(T)|_{l_{p}}^{-1}\cdot\mathbf{v}(y+T),\quad\mbox{for }y\in(-\infty,0].
    \end{equation*}
    \item[(3)] When $q_{-}>p$ (and $a=-\infty$), as $T\to-\infty$, we let
    \begin{equation*}
        \mathbf{v}_{T}(y):=|T|^{-\frac{p}{p-q_{-}}}\cdot\mathbf{v}(|T|\cdot y),\quad\mbox{for }y\in(-\infty,-1].
    \end{equation*}
\end{itemize}
For simplicity, we denote
\[
I=[0,1],\quad I=(-\infty,0],\quad\text{and}\quad I=(-\infty,-1]
\]
in cases (1), (2), and (3), respectively.
\end{definition}

\begin{lemma}\label{lem. v_T(y) asymptotic behavior}
    We have the following asymptotic estimates (independent of $T$):
\begin{itemize}
    \item[(1)] When $q_{-}<p$, then
    \begin{equation*}
        |\mathbf{v}_{T}(y)|\sim y^{\frac{p}{p-q_{-}}}\mbox{ and }|\nabla_{y}\mathbf{v}_{T}(y)|\sim y^{\frac{q}{p-q_{-}}},\quad\mbox{for }y\in I.
    \end{equation*}
    \item[(2)] When $q_{-}=p$, then there exist two uniform constants $0<\theta_{1}\leq \theta_{2}$ such that
    \begin{equation*}
        (\theta_{2})^{y}\lesssim|\mathbf{v}_{T}(y)|\lesssim(\theta_{1})^{y}\mbox{ and }(\theta_{2})^{y}\lesssim|\nabla_{y}\mathbf{v}_{T}(y)|\lesssim(\theta_{1})^{y},\quad\mbox{for }y\in I.
    \end{equation*}
    \item[(3)] When $q_{-}>p$, then
    \begin{equation*}
        |\mathbf{v}_{T}(y)|\sim|y|^{\frac{p}{p-q_{-}}}\mbox{ and }|\nabla_{y}\mathbf{v}_{T}(y)|\sim|y|^{\frac{q}{p-q_{-}}},\quad\mbox{for }y\in I.
    \end{equation*}
\end{itemize}
\end{lemma}
\begin{proof}
These estimates are direct consequences of Lemma~\ref{lem. (non-)existence of free boundary}
and the speed identity \eqref{eq. speed identity}.
\end{proof}

Since $$W(\mathbf{u})=W_{-}(\mathbf{v})\cdot\Big(1+E_{-}(\mathbf{v})\Big)$$ and  $W_{-}(\mathbf{v})$ is homogeneous of degree $q_{-}$, the following rescaled equations are satisfied by $\mathbf{v}_{T}(y)$ in $I$. The proof is a direct application of the chain rule, and we omit the proof.
\begin{lemma}\label{lem. rescaled equation for v_T}
    Assume that (A1)-(A5) hold for $W(\cdot)$. Let $\mathbf{u}(x)$ be a heteroclinical solution of \eqref{eq. GP equation}, and define $\mathbf{v}$ and $\mathbf{v}_{T}$ as above.
Then the following rescaled equations hold.
\begin{itemize}
    \item[(1)] If $q_{-}<p$, then
    \begin{align*}
        \Delta_{p,y}v^{i}_{T}(y)=&\Big(1+E_{-}(T^{\frac{p}{p-q_{-}}}\cdot\mathbf{v}_{T})\Big)\cdot D_{i}W_{-}(\mathbf{v}_{T})\\
        &+T^{\frac{p}{p-q_{-}}}\cdot W_{-}(\mathbf{v}_{T})\cdot D_{i}E_{-}(T^{\frac{p}{p-q_{-}}}\cdot\mathbf{v}_{T}).
    \end{align*}
    \item[(2)] If $q_{-}=p$, then
    \begin{align*}
        \Delta_{p,y}v^{i}_{T}(y)=&\Big(1+E_{-}(|\mathbf{v}(T)|_{l_{p}}\cdot\mathbf{v}_{T})\Big)\cdot D_{i}W_{-}(\mathbf{v}_{T})\\
        &+|\mathbf{v}(T)|_{l_{p}}\cdot W_{-}(\mathbf{v}_{T})\cdot D_{i}E_{-}(|\mathbf{v}(T)|_{l_{p}}\cdot\mathbf{v}_{T}).
    \end{align*}
    \item[(3)] If $q_{-}>p$, then
    \begin{align*}
        \Delta_{p,y}v^{i}_{T}(y)=&\Big(1+E_{-}(|T|^{\frac{p}{p-q_{-}}}\cdot\mathbf{v}_{T})\Big)\cdot D_{i}W_{-}(\mathbf{v}_{T})\\
        &+|T|^{\frac{p}{p-q_{-}}}\cdot W_{-}(\mathbf{v}_{T})\cdot D_{i}E_{-}(|T|^{\frac{p}{p-q_{-}}}\cdot\mathbf{v}_{T}).
    \end{align*}
\end{itemize}
Here, $D_{i}W_{-}(\cdot)$ and $D_{i}E_{-}(\cdot)$ are just two functions of the input (i.e., no chain rule is included).
\end{lemma}

For the sequence $\mathbf{v}_{T}(y)$ where $y\in I$, we next establish the following compactness result.
\begin{lemma}[Compactness]\label{Compactness}
For any sequence $T_{k}\to a+$, there exist a subsequence (still denoted by $T_{k}$) and a limiting function $\mathbf{v}_{a+}$ such that
\begin{equation}\label{eq. limiting equation for the blow-up/down limit}
    \Delta_{p,y}\mathbf{v}_{a+}(y)=D_{i}W_{-}(\mathbf{v}_{a+}),\quad\mbox{for all }y\in I,
\end{equation}
and
\[
\mathbf{v}_{T_k}\to\mathbf{v}_{a+}
\quad\text{in }C^{1,\epsilon}(I).
\]
Moreover,
\begin{equation*}
    \lim_{k\to\infty}\int_{y\in I}\Big\{\sum_{i=1}^{m}\frac{|\nabla_{y}v^{i}_{T_{k}}|^{p}}{p}+W_{-}(\mathbf{v}_{T_{k}})\Big\}dy=\int_{y\in I}\Big\{\sum_{i=1}^{m}\frac{|\nabla_{y}v^{i}_{a+}|^{p}}{p}+W_{-}(\mathbf{v}_{a+})\Big\}dy.
\end{equation*}
\end{lemma}
\begin{proof}
    This follows from Lemma~\ref{lem. v_T(y) asymptotic behavior}, Lemma~\ref{lem. rescaled equation for v_T}, and from the interior $C^{1,\epsilon}$ estimate of the $p$-Laplacian equation.
\end{proof}

The following key lemma characterizes the homogeneity of the limiting function
$\mathbf{v}_{a+}$. Its proof is quite involved as it requires a Weiss typed monotonicity formula with an additional error term. We will postpone its proof to the next section, and temporarily take it for granted in this section.
\begin{lemma}[Homogeneity of the limit]\label{lem. homogeneity of the limit}
    The limiting function $\mathbf{v}_{a+}$ takes the form
    \begin{equation*}
        \mathbf{v}_{a+}(y)=\mu(y)\cdot\mathbf{d},
    \end{equation*}
    where $\mu(y)$ is a scalar function and $\mathbf{d}\in\mathbb{R}^{m}$ is a constant vector. Moreover, all components of $\mathbf{d}$ are positive and uniformly comparable, meaning that there exists a constant $C>0$ independent of the choice of the sequence $T_{k}\to a+$, such that
    \begin{equation}\label{eq. d's components are uniformly comparable}
        C^{-1} d^{j}\leq d^{i}\leq C d^{j},\quad\mbox{for all }1\leq i,j\leq m.
    \end{equation}
    The scalar function $\mu(y)$ takes the following form:
    \begin{equation}\label{formmu}
        \mu(y)=\left\{\begin{aligned}
            &y^{\frac{p}{p-q_{-}}}&\mbox{for }&y\in[0,1],&\mbox{if }&q_{-}<p,\\
            &\theta^{y}&\mbox{for }&y\in(-\infty,0],&\mbox{if }&q_{-}=p,\\
            &|y|^{\frac{p}{p-q_{-}}}&\mbox{for }&y\in(-\infty,-1],&\mbox{if }&q_{-}>p.\\
        \end{aligned}\right.
    \end{equation}
    Here, $\theta\in[\theta_{1},\theta_{2}]$ where $\theta_{1}$ and $\theta_{2}$ are constants mentioned in Lemma~\ref{lem. (non-)existence of free boundary}. The constant exponent $b$ and the constant vector $\mathbf{d}$ might depend on the choice of $T_{k}\to a+$.
\end{lemma}

Assuming that Lemma~\ref{lem. homogeneity of the limit} is correct, we have the following corollary.
\begin{corollary}[Almost homogeneity]\label{cor3.1}
    For all $1\leq i,j\leq m$, it holds that
    \begin{equation*}
        \lim_{x\to a+}\Big(\frac{|u_{x}^{i}|^{p-2}u_{x}^{i}}{D_{i}W(\mathbf{u})}:\frac{|u_{x}^{j}|^{p-2}u_{x}^{j}}{D_{j}W(\mathbf{u})}\Big)=1,\quad\mbox{and }\limsup_{x\to a+}\frac{u^{i}_{x}}{u^{j}_{x}}<+\infty.
    \end{equation*}
    and 
     \begin{equation*}
        \lim_{x\to b-}\Big(\frac{|u_{x}^{i}|^{p-2}u_{x}^{i}}{D_{i}W(\mathbf{u})}:\frac{|u_{x}^{j}|^{p-2}u_{x}^{j}}{D_{j}W(\mathbf{u})}\Big)=1,\quad\mbox{and }\limsup_{x\to b-}\frac{u^{i}_{x}}{u^{j}_{x}}<+\infty.
    \end{equation*}
\end{corollary}
The proof of Corollary~\ref{cor3.1} will also be postponed to the next section.

\section{Weiss typed functions and almost homogeneity}
In this section, we apply the Weiss monotonicity formula \cite{W99} to study the almost homogeneity of the heteroclinical solution near the potential wells. This approach allows us to extract quantitative information on the asymptotic profile of the solution and derive the sharp decay behavior. As an application, we will provide the proof of Lemma~\ref{lem. homogeneity of the limit} and Corollary~\ref{cor3.1}.
\subsection{Weiss-typed monotonicity with an error term} We first define three types of Weiss functions and prove their almost monotonicity.
\begin{definition}\label{def. Weiss original definition}
Define the Weiss-typed function as follows:
\begin{itemize}
    \item \textbf{Case 1: $q_{-}<p$.} Like in Subsection~\ref{subsec. Blow-up or blow-down}, we assume $a=0$. As $T\to0+$, set
    \begin{align*}
        \Delta_{p,y}v^{i}_{T}(y)=&\Big(1+E_{-}(T^{\frac{p}{p-q_{-}}}\cdot\mathbf{v}_{T})\Big)\cdot D_{i}W_{-}(\mathbf{v}_{T})\\
        &+T^{\frac{p}{p-q_{-}}}\cdot W_{-}(\mathbf{v}_{T})\cdot D_{i}E_{-}(T^{\frac{p}{p-q_{-}}}\cdot\mathbf{v}_{T}).
    \end{align*}
    \item \textbf{Case 2: $q_{-}=p$.} As $T\to-\infty$, set
    \begin{equation*}
        \mathcal{W}(T)=e^{-pf(T)} \int_{-\infty}^T \left(\sum_{i=1}^m \frac{|\nabla v^i(x)|^p}{p}+W_{-}(\mathbf{v}(x))\right)dx,
    \end{equation*}
    where $f(T)$ is a function satisfying
    \begin{equation}\label{fS}
        \sum_{i=1}^m |v^i(T)|^p=e^{pf(T)}.
    \end{equation}
    \item \textbf{Case 3: $q_{-}>p$.} As $T\to-\infty$, set 
       \begin{align*}
        \mathcal{W}(T)=&|T|^{\frac{pq_{-}+p-q_{-}}{q_{-}-p}}\int_{-\infty}^{T} \left(\sum_{i=1}^m \frac{|\nabla v^i(x)|^p}{p}+W_{-}(\mathbf{v}(x)) \right) dx\\
        &-\left(\frac{p}{q_{-}-p}\right)^{p-1}\cdot |T|^{\frac{p^2}{q_{-}-p}} \sum_{i=1}^{m}\frac{|v^i(T)|^p}{p}.
    \end{align*}
\end{itemize}
\end{definition}

We have the following two lemmas for the rescaled function $\mathbf{v}_{T}(y)$ defined in Subsection~\ref{subsec. Blow-up or blow-down}. The proof follows from a direct computation via the chain rule and also from the homogeneity of $W_{-}(\mathbf{v})$. We omit their proofs and leave them to the readers.
\begin{lemma}\label{lem. rescaled Weiss function}
    Let $\mathbf{v}_{T}(y)$ be defined as in Subsection~\ref{subsec. Blow-up or blow-down} for $y\in I$. Then the Weiss functions defined in Definition~\ref{lem. v_T(y) asymptotic behavior} have the following alternative forms:
    \begin{itemize}\item \textbf{Case 1: $q_{-}<p$.} We have
    \begin{equation*}
        \mathcal{W}(T)=\int_0^1 \left(\sum_{i=1}^m \frac{|\nabla_y v_T^i(y)|^p}{p}+W_{-}(\mathbf{v}_T(y)) \right) dy-\left(\frac{p}{p-q_{-}}\right)^{p-1}\cdot\sum_{i=1}^{m}\frac{|v^i_T(1)|^p}{p}.
    \end{equation*}
    \item \textbf{Case 2: $q_{-}=p$.} We have
    \begin{equation*}
        \mathcal{W}(T)= \int_{-\infty}^0 \left(\sum_{i=1}^m \frac{|\nabla v^i_T(x)|^p}{p}+W_{-}(\mathbf{v}_T(y)) \right) dy.
    \end{equation*}
    \item \textbf{Case 3: $q_{-}>p$.} We have
 \begin{equation*}
        \mathcal{W}(T)= \int_{-\infty}^{-1} \left(\sum_{i=1}^m \frac{|\nabla_y v_T^i(y)|^p}{p}+W_{-}(\mathbf{v}_T(y)) \right) dy-\left(\frac{p}{q_{-}-p}\right)^{p-1}\cdot\sum_{i=1}^{m}\frac{|v^i_T(-1)|^p}{p}.
    \end{equation*}
    \end{itemize}
\end{lemma}

\begin{lemma}\label{lem. v dot T (y)}
    Let $\mathbf{v}_{T}(y)$ be defined as in Subsection~\ref{subsec. Blow-up or blow-down} for $y\in I$. Let the derivative of $\mathbf{v}_{T}$ in $T$ be denoted as $\dot{\textbf{v}}_{T}(y)=\big(\dot{v}_{T}^{1}(y),\cdots,\dot{v}_{T}^{m}(y)\big)$, where
    \begin{equation*}
        \dot{\textbf{v}}_{T}(y)=\frac{d}{dT}\mathbf{v}_{T}(y)=\lim_{t\to0}\frac{\mathbf{v}_{T+t}(y)-\mathbf{v}_{T}(y)}{t}.
    \end{equation*}
    Then, we have the following identities:
    \begin{itemize}\item \textbf{Case 1: $q_{-}<p$.} We have
    \begin{equation*}
        \dot{\textbf{v}}_{T}(y)=\frac{1}{T}\left(y\nabla_y \mathbf{v}_T(y)-\frac{p}{p-q_{-}} \mathbf{v}_T(y)  \right).
    \end{equation*}
    \item \textbf{Case 2: $q_{-}=p$.} We have
    \begin{equation*}
\dot{\mathbf{v}}_T(y)=\nabla_y \mathbf{v}_T(y)-f'(T)\mathbf{v}_T(y),
\end{equation*}
    \item \textbf{Case 3: $q_{-}>p$.} We have
   \begin{equation*}
        \dot{\textbf{v}}_{T}(y)=\frac{1}{|T|}\left(-y\nabla_y \mathbf{v}_T(y)-\frac{p}{p-q_{-}} \mathbf{v}_T(y)  \right).
    \end{equation*}
\end{itemize}
\end{lemma}

We now state and prove the following Weiss-typed monotonicity formula. In comparison to \cite{W99}, our monotonicity formula contains a small negative error term. This additional error term arises from
the fact that the potential function $W(\cdot)$ is not exactly homogeneous near
the potential wells $\pm\mathbf{e}$.
\begin{lemma}\label{lem. Weiss Monotonicity}
Let $\mathcal{W}(T)$ be   the Weiss-type function defined as in Lemma \ref{lem. rescaled Weiss function}. Then, there exists a uniform constant $C=C\big(p,W(\cdot)\big)$ and a continuous $p$-degree homogeneous function
\begin{equation*}
    \mathcal{P}(\mathbf{a},\mathbf{b}):\mathbb{R}^{m}\times(\mathbb{R}^{m}\setminus\{0\})\to[0,+\infty),
\end{equation*}
which also depends on the exponents $(p,q_{-})$, such that the following holds:
\begin{itemize}
    \item[(1)] If $q_{-}<p$, then $\mathcal{P}(\mathbf{a},\mathbf{b})=0$ if and only if $\mathbf{a}=\frac{p}{p-q_{-}}\cdot\mathbf{b}$, and
    \begin{equation*}
        \frac{d}{dt} \mathcal{W}(T)\geq\frac{1}{T}\cdot\mathcal{P}\Big(\nabla_{y}\mathbf{v}_{T}(1),\mathbf{v}_{T}(1)\Big)-C T^{\frac{q_-}{p-q_-}}\geq-C T^{\frac{q_-}{p-q_-}}.
    \end{equation*}
    \item[(2)] If $q_{-}=p$, then $\mathcal{P}(\mathbf{a},\mathbf{b})=0$ if and only if $\mathbf{a}\parallel\mathbf{b}$, and
    \begin{equation*}
        \frac{d}{dt} \mathcal{W}(T)\geq\mathcal{P}\Big(\nabla_{y}\mathbf{v}_{T}(0),\mathbf{v}_{T}(0)\Big)-C (\theta_1)^T\geq-C (\theta_1)^T.
    \end{equation*}
    \item[(3)] If $q_{-}>p$, then $\mathcal{P}(\mathbf{a},\mathbf{b})=0$ if and only if $\mathbf{a}=\frac{p}{q_{-}-p}\cdot\mathbf{b}$, and
    \begin{equation*}
        \frac{d}{dt} \mathcal{W}(T)\geq\frac{1}{|T|}\cdot\mathcal{P}\Big(\nabla_{y}\mathbf{v}_{T}(-1),\mathbf{v}_{T}(-1)\Big)-C |T|^{\frac{q_-}{p-q_-}}\geq-C |T|^{\frac{q_-}{p-q_-}}.
    \end{equation*}
\end{itemize}
\end{lemma}
\begin{proof}
The proof of Lemma~\ref{lem. Weiss Monotonicity} (especially Case 1) is inspired by the argument in  \cite{W99}.
We provide the details for
Cases 1 and 2. The proof of Case 3 is analogous to that of Case 1 and is
therefore omitted.

{\bf{Case 1:  $q_{-}<p$.}} Differentiating the expression in Lemma~\ref{lem. rescaled Weiss function} yields
\begin{align*}
    \frac{d}{dT}\mathcal{W}(T) =&\int_0^1 \sum_{i=1}^m |\nabla_y v_T^i(y)|^{p-2}<\nabla_y v_T^i(y), \nabla_y \dot{v}_T^i(y)> dy\\
    &+ \sum_{i=1}^m \int_0^1 D_i W_{-}(\mathbf{v}_T(y)) \dot{v}_T^i(y) dy-(\frac{p}{p-q_{-}})^{p-1}\cdot\sum_{i=1}^m | v_T^i(1)|^{p-2} v_T^i(1) \dot{v}_T^i(1).
\end{align*}
After integration by parts, we obtain
\begin{align*}
    \frac{d}{dT}\mathcal{W}(T)=&\sum_{i=1}^{m}\int_{0}^{1}\dot{v}_{T}^{i}(y)\cdot\Big\{-\Delta_{p,y}v_{T}^{i}(y)+D_{i}W_{-}(\mathbf{v}_T(y))\Big\}dy\\
    &+\sum_{i=1}^{m}\dot{v}_{T}^{i}(1)\Big\{|\nabla_{y}v_{T}^{i}(1)|^{p-2}\nabla_{y}v_{T}^{i}(1)-(\frac{p}{p-q_{-}})^{p-1}\cdot|v_{T}^{i}(1)|^{p-2}v_{T}^{i}(1)\Big\}\\
    :=&I_1+I_2.
\end{align*}

It follows from Lemma~\ref{lem. v_T(y) asymptotic behavior} and Lemma~\ref{lem. v dot T (y)} that
\begin{equation*}
    |\mathbf{v}_{T}(y)|\leq C\cdot y^{\frac{p}{p-q_{-}}}\mbox{ and }|\dot{\mathbf{v}}_{T}(y)|\leq \frac{C}{T}\cdot y^{\frac{p}{p-q_{-}}},\quad\mbox{for }y\in I=[0,1].
\end{equation*}
Then, by the rescaled equation for $\mathbf{v}_{T}$ in Lemma~\ref{lem. rescaled equation for v_T} (1) and by the assumption (A5) for the potential $W(\cdot)$, we have
\begin{equation*}
    \Big|\Delta_{p,y}v^{i}_{T}(y)-D_{i}W_{-}(\mathbf{v}_{T})\Big|\leq C\cdot(T\cdot y^{q_{-}})^{\frac{p}{p-q_{-}}}.
\end{equation*}
Therefore,
\begin{equation*}
    |I_{1}|\leq\int_{0}^{1}\Big(\frac{C}{T}\cdot y^{\frac{p}{p-q_{-}}}\Big)\cdot\Big(C\cdot(T\cdot y^{q_{-}})^{\frac{p}{p-q_{-}}}\Big)dy\leq C\cdot T^{\frac{q_{-}}{p-q_{-}}}.
\end{equation*}

The term $I_{2}$ is to be written as $I_{2}=\frac{1}{T}\cdot\mathcal{P}\Big(\nabla_{y}\mathbf{v}_{T}(1),\mathbf{v}_{T}(1)\Big)$, where
\begin{equation*}
    \mathcal{P}(\mathbf{a},\mathbf{b})=\sum_{i=1}^{m}\Big(a^{i}-\frac{p}{p-q_{-}}b^{i}\Big)\Big(|a^{i}|^{p-2}a^{i}-(\frac{p}{p-q_{-}})^{p-1}|b^{i}|^{p-2}b^{i}\Big)\geq0.
\end{equation*}
One can easily verify that $\mathcal{P}(\mathbf{a},\mathbf{b})=0$ if and only if $\mathbf{a}=\frac{p}{p-q_{-}}\mathbf{b}$.

{\bf{Case 2:  $q_{-}=p$.}} Similar to the previous case, we get
$$
\begin{aligned}
\frac{d}{dT}\mathcal{W}(T)=&-\int_{-\infty}^0 \dot{v}_T^i(y)   \left(-\Delta_{p,y}v_T^i(y) +D_iW_{-}( v_T^i(y))      \right)dy \\
&+\sum_{i=1}^m \dot{v}_T^i(0) |\nabla_y v_T^i(0)|^{p-2} \nabla_y v_T^i(0)\\
:=&J_1+J_2.
\end{aligned}
$$

We first estimate the term $J_1$. By \eqref{fS}, one has
\begin{equation*}
    \sum_{i=1}^m |v_T^i(0)|^p= \sum_{i=1}^m e^{-pf(T)}{|v^i(T)|^p}=1.
\end{equation*}
In addition, differentiating both sides of \eqref{fS} with respect to $T$ yields that
\begin{equation}\label{eq. f'(T) expression}
    f'(T)=e^{-pf(T)}\sum_{i=1}^m |v^i(T)|^{p-1}  \nabla_x v^i(T)=\sum_{i=1}^m |v_T^i(0)|^{p-1}\nabla_y v^i_T(0).
\end{equation}
It then follows from Lemma~\ref{lem. v_T(y) asymptotic behavior} that $|f'(T)|\leq C$. Hence, by Lemma~\ref{lem. v dot T (y)}, we have
\begin{equation*}
    |\mathbf{v}_{T}(y)|\le C (\theta_{1})^{y}\mbox{ and } |\dot{\mathbf{v}}_{T}(y)|\leq C(\theta_{1})^{y},\quad\mbox{for }y\in I=(-\infty,0].
\end{equation*}
Then, by the rescaled equation for $\mathbf{v}_{T}$ in Lemma~\ref{lem. rescaled equation for v_T} (2), Lemma \ref{lem. (non-)existence of free boundary} and  the assumption (A5) for the potential $W(\cdot)$, we have
\begin{equation*}
    \Big|\Delta_{p,y}v^{i}_{T}(y)-D_{i}W_{-}(\mathbf{v}_{T})\Big|\leq C (\theta_1)^{T+q_-y}.
\end{equation*}
This implies the following estimate for $|J_{1}|$:
\begin{equation*}
    |J_{1}|\leq\int_{-\infty}^0\Big(C(\theta_{1})^{y}\Big)\cdot\Big(C(\theta_1)^{T+q_-y}\Big)dy\leq C(\theta_1)^T.
\end{equation*}

Next we prove the positivity of $J_{2}$. By Lemma~\ref{lem. v dot T (y)} and the expression \eqref{eq. f'(T) expression}, we have
\begin{equation*}
J_2=\sum_{i=1}^{m}\left(\nabla_{y}v_{T}^{i}(0)-f'(T)v_{T}^{i}(0)\right)|\nabla_{y}v_{T}^{i}(0)|^{p-2}\nabla_{y}v_{T}^{i}(0)=\mathcal{P}\Big(\nabla_{y}\mathbf{v}_{T}(0),\mathbf{v}_{T}(0)\Big),
\end{equation*}
where
\begin{equation*}
    \mathcal{P}(\mathbf{a},\mathbf{b})=|\mathbf{a}|_{l_{p}}^{p}-\frac{1}{|\mathbf{b}|_{l_{p}}^{p}}\sum_{i=1}^{m}|a^{i}|^{p-2}a^{i}b^{i}\cdot\sum_{j=1}^{m}|b^{j}|^{p-2}b^{j}a^{j}.
\end{equation*}
By the H\"older inequality, one has
\begin{equation*}
    \sum_{i=1}^{m}|a^{i}|^{p-2}a^{i}b^{i}\leq|\mathbf{a}|_{l_{p}}^{p-1}\cdot|\mathbf{b}|_{l_{p}}\mbox{ and }\sum_{j=1}^{m}|b^{j}|^{p-2}b^{j}a^{j}\leq|\mathbf{b}|_{l_{p}}^{p-1}\cdot|\mathbf{a}|_{l_{p}}.
\end{equation*}
Then, we see that $\mathcal{P}(\mathbf{a},\mathbf{b})\geq0$ and equality holds if and only if $\mathbf{a}\parallel\mathbf{b}$.

{\bf{Case 3:  $q_{-}>p$.}} The proof is similar to Case 1, and we omit the details. 
\end{proof}

\subsection{Analysis of the blow-up/down limits}
In this subsection, we apply the Weiss-typed monotonicity formulas to classify the blow-up/down limits. As a consequence, we are able to prove Lemma~\ref{lem. homogeneity of the limit} and Corollary~\ref{cor3.1} below.
\begin{proof}[Proof of Lemma~\ref{lem. homogeneity of the limit}]
The proof is divided into three steps.

\textbf{Step 1: Consequence of the Weiss-typed monotonicity formula.} We claim that $\mathbf{v}_{a+}$ satisfies the following homogeneity properties:
\begin{itemize}
    \item[(1)] If $q_{-}<p$ and $a=0$, then
    \begin{equation*}
        y\cdot\nabla_{y}\mathbf{v}_{a+}(y)\equiv\frac{p}{p-q_{-}}\cdot\mathbf{v}_{a+}(y)\quad\mbox{for all }y\in(0,1];
    \end{equation*}
    \item[(2)] If $q_{-}=p$ and $a=-\infty$, then
    \begin{equation*}
        \nabla_{y}\mathbf{v}_{a+}(y)\parallel\mathbf{v}_{a+}(y)\quad\mbox{for all }y\in(-\infty,0];
    \end{equation*}
    \item[(3)] If $q_{-}>p$ and $a=-\infty$, then
    \begin{equation*}
        y\cdot\nabla_{y}\mathbf{v}_{a+}(y)\equiv\frac{p}{q_{-}-p}\cdot\mathbf{v}_{a+}(y)\quad\mbox{for all }y\in(-\infty,-1].
    \end{equation*}
\end{itemize}

We will only prove the claim in the first case ($q_{-}<p$). 
Let $T_k\to a+=0+$. By Lemma~\ref{Compactness}, after passing to a subsequence
(still denoted by $T_k$), we have
\[
\mathbf v_{T_k}\to\mathbf v_{a+}
\quad\text{in }C^{1,\epsilon}_{\rm loc}(I),
\]
where $\mathbf v_{a+}$ satisfies the limiting equation
\eqref{eq. limiting equation for the blow-up/down limit}. Moreover, by
Lemma~\ref{lem. v_T(y) asymptotic behavior},
\begin{equation*}
|\mathbf v_T(y)|\leq C y^{\frac{p}{p-q_-}}\mbox{ and }|\nabla_{y}\mathbf{v}_{T}(y)|\leq C y^{\frac{q_-}{p-q_-}}\quad\mbox{for all }y\in(0,1),
\end{equation*}
where the constant $C$ is independent of $T$. 

On the other hand, since $W_{-}$ is homogeneous of degree $q_{-}$, then it follows from the expression in Lemma~\ref{lem. rescaled Weiss function} that $|\mathcal{W}(T)|$ is uniformly bounded as $T\to a+$. 
Therefore,
by Lemma \ref{lem. Weiss Monotonicity}, we obtain
\begin{equation*}
    \frac{1}{T}\mathcal{P}\left(\nabla_{y}\mathbf{v}_{T}(1),\mathbf{v}_{T}(1)\right)\leq\frac{d}{dT}\mathcal{W}(T)+C T^{\frac{q_-}{p-q_-}} .
\end{equation*}
For any fixed $T_k<T_l$, integrating the inequality in Lemma~\ref{lem. Weiss Monotonicity} over $(T_k,T_l)$ gives
\begin{equation*}
\int_{T_k}^{T_l}\frac{1}{T}\mathcal{P}\big(\nabla_{y}\mathbf{v}_{T}(1),\mathbf{v}_{T}(1)\big)dT\leq\mathcal{W}(T_{l})-\mathcal{W}(T_{k})+CT_{l}^{\frac{p}{p-q_{-}}}.
\end{equation*}
Since $|\mathcal{W}(T)|$ is uniformly bounded, we conclude that
\begin{equation*}
    \int_{0}^{1}\frac{1}{T}\mathcal{P}\left(\nabla_{y}\mathbf{v}_{T}(1),\mathbf{v}_{T}(1)\right)dT<+\infty.
\end{equation*}
It then follows from the Dominated Convergence Theorem that
\begin{equation}\label{eq. sending T_k to zero, integration of P}
    \lim_{k\to\infty}\int_{0}^{T_{k}}\frac{1}{T}\mathcal{P}\left(\nabla_{y}\mathbf{v}_{T}(1),\mathbf{v}_{T}(1)\right)dT=0.
\end{equation}
Notice that the rescaled functions defined in
Definition~\ref{def. blow-up/down process, v_T construction} satisfy the scaling
property
\[
(\mathbf v_{T_k})_S(z)=\mathbf v_{T_kS}(z).
\]
Taking $S=T/T_k$ in \eqref{eq. sending T_k to zero, integration of P}, we obtain
\begin{equation*}
    \lim_{k\to\infty}\int_{0}^{1}\frac{1}{S}\mathcal{P}\Big(\nabla_{y}(\mathbf{v}_{T_{k}})_{S}(1),(\mathbf{v}_{T_{k}})_{S}(1)\Big)dS=0.
\end{equation*}
Since $\mathbf{v}_{T_{k}}\to\mathbf{v}_{a+}$ in the $C^{1,\epsilon}([0,1])$ sense, it follows that for every $h>0$,
\begin{equation*}
    \lim_{k\to\infty}\sup_{S\in[h,1]}\Big|\mathcal{P}\Big(\nabla_{y}(\mathbf{v}_{T_{k}})_{S}(1),(\mathbf{v}_{T_{k}})_{S}(1)\Big)-\mathcal{P}\Big(\nabla_{y}(\mathbf{v}_{a+})_{S}(1),(\mathbf{v}_{a+})_{S}(1)\Big)\Big|=0.
\end{equation*}
As a result, one has
\begin{equation*}
    \int_{h}^{1}\frac{1}{S}\mathcal{P}\Big(\nabla_{y}(\mathbf{v}_{a+})_{S}(1),(\mathbf{v}_{a+})_{S}(1)\Big)dS=0,\quad\mbox{for every }h>0.
\end{equation*}
Consequently, one has \[
\mathcal P
\Big(
\nabla_y(\mathbf v_{a+})_S(1),
(\mathbf v_{a+})_S(1)
\Big)
\equiv0,
\qquad S\in(0,1].
\]. Using the equality condition in Lemma~\ref{lem. Weiss Monotonicity}, we have proven the first case of the claim. The other two cases can be proven analogously.

\textbf{Step 2: Homogeneity of $\mathbf{v}_{a+}$.} 
By \textbf{Step 1}, the limiting function $\mathbf v_{a+}$ satisfies the following
homogeneity properties:
\begin{itemize}
    \item[(1)] If $q_{-}<p$ and $a=0$, then
    \begin{equation*}
        \mathbf{v}_{a+}(y)\equiv y^{\frac{p}{p-q_{-}}}\cdot\mathbf{d}\quad\mbox{for all }y\in(0,1];
    \end{equation*}
    \item[(2)] If  $q_{-}=p$ and $a=-\infty$, then
    \begin{equation*}
        \mathbf{v}_{a+}(y)\equiv\varphi(y)\cdot\mathbf{d}\quad\mbox{for all }y\in(-\infty,0];
    \end{equation*}
    \item[(3)] If $q_{-}>p$ and $a=-\infty$, then
    \begin{equation*}
        \mathbf{v}_{a+}(y)\equiv|y|^{\frac{p}{p-q_{-}}}\cdot\mathbf{d}\quad\mbox{for all }y\in(-\infty,-1].
    \end{equation*}
\end{itemize}
Moreover, we intend to show below that $\varphi(y)$ takes the form $\theta^{y}$.

To see this, as $W_{-}$ is homogeneous of degree $p$ in case (2), it follows from \eqref{eq. limiting equation for the blow-up/down limit} that
\begin{equation*}
    \Delta_{p,y}\varphi(y)\cdot\mathbf{d}=\varphi(y)^{p-1}\cdot D_{i}W_{-}(\mathbf{d}).
\end{equation*}
As a result, $$\Delta_{p,y}\varphi(y)\equiv c\cdot\varphi(y)^{p-1}$$ for some constant $c$. Multiplying this equation by $\varphi'(y)$  and  integrating gives that
\begin{equation}\label{eq. 1st order ODE of varphi}
    (\varphi'(y))^{p}-c_{1}\varphi(y)^{p}=c_{2}.
\end{equation}
Recall that it follows from Lemma~\ref{lem. v_T(y) asymptotic behavior} that $$\varphi'(-\infty)=\varphi(-\infty)=0,$$ then we see that $c_{2}=0$ in \eqref{eq. 1st order ODE of varphi} and that $$\varphi(y)=c\cdot\theta^{y}.$$ By replacing $\varphi(y)$ with $\theta^{y}$ and $\mathbf{d}$ with $c\cdot\mathbf{d}$, one has that $$\varphi(y)=\theta^{y}$$ for some $\theta>1$ in the second case ($q_{-}=p$).

\textbf{Step 3: Uniform comparability of the components of $\mathbf{d}$.} 

By taking $y=1$ in case (1) of the conclusion of \textbf{Step 2}, we have $$\mathbf{d}=\mathbf{v}_{a+}(1).$$ Moreover, by Lemma \ref{lem. v_T(y) asymptotic behavior}, there exists a constant $C>1$ such that $$C^{-1}\leq |\mathbf{v}_{T_k}(1)|\leq C.$$ Passing to the limit gives $$C^{-1}\leq|\mathbf{d}|\leq C,$$ that is, $$|\mathbf{d}|\sim1.$$ Similarly, one also has $|\mathbf{d}|\sim1$ in other two cases. Moreover, since $v^{i}(x)\geq0$ for all $i\in\{1,\cdots,m\}$, we conclude that
\begin{equation*}
    \mathbf{d}\in[0,+\infty)^{m}\quad\mbox{and }|\mathbf{d}|\sim1.
\end{equation*}

It suffices to show that $d^{i}>0$ for all $i\in\{1,\cdots,m\}$ and \eqref{eq. d's components are uniformly comparable} holds for every possible subsequence $T_{k}\to a+$. To see this, it follows from the homogeneity property of $\mathbf{v}_{a+}$ in \textbf{Step 2} and from \eqref{eq. limiting equation for the blow-up/down limit} (where $W_{-}$ is a $q_{-}$-degree homogeneous function) that
\begin{equation}\label{eq. algebraic equation}
    \Big(|d^{1}|^{p-2}d^{1},\cdots,|d^{m}|^{p-2}d^{m}\Big)\parallel DW_{-}(\mathbf{d}),\quad\mbox{i.e. }D_{i}W_{-}(\mathbf{d})=\lambda|d^{i}|^{p-2}d^{i}.
\end{equation}
We claim that if $\mathbf{d}\in[0,+\infty)^{m}\cap\partial B_{1}$ satisfies $d^{i}=0$ for some $i\in\{1,\cdots,m\}$, then $\mathbf{d}$ is not a solution to \eqref{eq. algebraic equation}.

In order to prove the claim, suppose that $\mathbf{d}\in[0,+\infty)^{m}\cap\partial B_{1}$ is a solution to \eqref{eq. algebraic equation} with $d^{1}=0$. Then, we must have $$D_{1}W_{-}(\mathbf{d})=0.$$ By the homogeneity of $W_{-}(\cdot)$, it holds that \[
D_1W_{-}(t\mathbf d)=0
\qquad\text{for all }t>0.
\] 
Differentiating with respect to $t$ gives
\begin{equation*}
    \sum_{i=1}^{m}D_{1i}W_{-}(\mathbf{d})\cdot d^{i}=0.
\end{equation*}
However, since $$d^{1}=0, \,\mathbf{d}\in[0,+\infty)^{m}\cap\partial B_{1},$$ it then follows from the assumption (A4) that indeed $$\displaystyle\sum_{i=1}^{m}D_{1i}W_{-}(\mathbf{d})\cdot d^{i}<0,$$ which is a contradiction. Hence, the claim is verified.

Moreover, it is easy to see that all solutions to the problem \eqref{eq. algebraic equation}, after being projected to $\mathbb{S}^{m-1}=\partial B_{1}$, form a closed set. Therefore, there exists a uniform constant $C$, such that all solutions $\mathbf{d}\in[0,+\infty)^{m}\cap\partial B_{1}$ to \eqref{eq. algebraic equation} must satisfy \eqref{eq. d's components are uniformly comparable}. This completes the proof of Lemma \ref{lem. homogeneity of the limit}.
\end{proof}
\begin{proof}[Proof of Corollary~\ref{cor3.1}]
We only prove the assertion as $x\to a+$. The proof for $x\to b-$ is analogous. Suppose that the conclusions in Corollary~\ref{cor3.1} fails, then there exists a sequence $x_{k}\to a+$ and a fixed pair $i\neq j\in\{1,\cdots,m\}$, such that
\begin{equation}\label{eq. counter example either or}
    \mbox{either }\liminf_{x\to a+}\Big(\frac{|u_{x}^{i}|^{p-2}u_{x}^{i}}{D_{i}W(\mathbf{u})}:\frac{|u_{x}^{j}|^{p-2}u_{x}^{j}}{D_{j}W(\mathbf{u})}\Big)>1,\quad\mbox{or }\lim_{x\to a+}\frac{u^{i}_{x}}{u^{j}_{x}}=+\infty.
\end{equation}

For such a subsequence $x_{k}$, we define $$\mathbf{v}=\mathbf{u}+\mathbf{e}$$ and set $$\mathbf{v}_{k}(y)=\mathbf{v}_{T_{k}}(y)\,\,\mbox{for}\,\, T_{k}=x_{k},$$ as in Definition~\ref{def. blow-up/down process, v_T construction}. It follows from Lemma~\ref{Compactness} and Lemma~\ref{lem. homogeneity of the limit} that there exists a subsequence of $\mathbf{v}_{k}(y)$ (still labeled as $\mathbf{v}_{k}$) converges to a homogeneous solution $\mathbf{v}_{a+}$ (which possibly depends on the choice of $x_{k}$) to \eqref{eq. limiting equation for the blow-up/down limit} in the $C^{1,\epsilon}(I)$ sense. Consequently, we have
\begin{equation}\label{eq. converge of counter example sequence}
    \lim_{k\to\infty}\nabla_{y}\mathbf{v}_{k}(1)=\nabla_{y}\mathbf{v}_{a+}(1),\quad\mbox{and }\lim_{k\to\infty}\mathbf{v}_{k}(1)=\mathbf{v}_{a+}(1).
\end{equation}
Notice that it follows from Lemma~\ref{lem. homogeneity of the limit} and \eqref{eq. algebraic equation} that for every $i,j\in\{1,\cdots,m\}$,
\begin{equation*}
    \frac{|\nabla_{y}v_{a+}^{i}(1)|^{p-2}\nabla_{y}v_{a+}^{i}(1)}{D_iW_{-}\big(\mathbf{v}_{a+}(1)\big)}=\frac{|\nabla_{y}v_{a+}^{j}(1)|^{p-2}\nabla_{y}v_{a+}^{j}(1)}{D_jW_{-}\big(\mathbf{v}_{a+}(1)\big)}\quad\mbox{and }\nabla_{y}v^{i}_{a+}(1)\leq C\nabla_{y}v^{j}_{a+}(1).
\end{equation*}
This, together with \eqref{eq. converge of counter example sequence}, implies that for every $i,j\in\{1,\cdots,m\}$, one has
\begin{equation*}
    \lim_{k\to\infty}\Big(\frac{|\nabla_{y}v_{k}^{i}(1)|^{p-2}\nabla_{y}v_{k}^{i}(1)}{D_iW_{-}\big(\mathbf{v}_{k}(1)\big)}:\frac{|\nabla_{y}v_{k}^{j}(1)|^{p-2}\nabla_{y}v_{k}^{j}(1)}{D_jW_{-}\big(\mathbf{v}_{k}(1)\big)}\Big)=1\mbox{ and }\nabla_{y}v^{i}_{k}(1)\leq C\nabla_{y}v^{j}_{k}(1).
\end{equation*}

By the definition of $\mathbf{v}$ and $\mathbf{v}_{k}$, we see that the property above contradicts the assumption \eqref{eq. counter example either or}. Hence, the proof of Corollary~\ref{cor3.1} is complete.
\end{proof}

\section{Monotonicity of the heteroclinical solution}

In this section, we establish the monotonicity of each component of the heteroclinical solution.
We first collect several basic properties of the underlying cooperative system, which will play an essential role in the proof of monotonicity.

\subsection{Basic properties for the cooperative system}
Since we have assumed that $W_{ij}<0$ for $i\neq j$ (except at the potential wells $\pm\mathbf{e}$), we have the following two important observations, which will be essential in proving the monotonicity of the solution.

\begin{lemma}\label{lem. boundary stability of pm e}
    Let $\mathbf{u}(x)$ satisfy \eqref{eq. GP equation} in the classical sense near $x=0$. If for some $i\in\{1,\cdots,m\}$ we have $$u^{i}(0)=-\Lambda_{i}\,\, (\mbox{or} \,\Lambda_{i}),$$ then $$\mathbf{u}(0)=-\mathbf{e}\, (\mbox{or }\,\mathbf{e},\, \mbox{ respectively}).$$
\end{lemma}
\begin{proof}
    Without loss of generality, we assume that
\[
u^1(0)=-\Lambda_1,
\]
but
\[
u^2(0)>-\Lambda_2.
\] 
We shall derive a contradiction. Since $u^{1}(x)\geq-\Lambda_{1}=u^{1}(0)$, we must have $$W_{1}(\mathbf{u}(0))=\Delta_{p}u^{1}(0)\geq0.$$
By (A3)-(A5), we know that $W_{1}(-\mathbf{e})=0$. By integration,
    \begin{align*}
        W_{1}(\mathbf{u}(0))=&\int_{0}^{1}\frac{d}{dt}W_{1}\Big(-\mathbf{e}+(\mathbf{u}(0)+\mathbf{e})\cdot t\Big)dt\\
        =&\sum_{i=1}^{m}\int_{0}^{1}W_{1i}\Big(-\mathbf{e}+(\mathbf{u}(0)+\mathbf{e})\cdot t\Big)\cdot(u^{i}(0)+\Lambda_{i})dt.
    \end{align*}
    Recall that $W_{ij}<0$ for $i\neq j$ except possibly at $\pm\mathbf{e}$, and notice that
    \begin{equation*}
        u^{i}(0)+\Lambda_{i}=0,\quad\mbox{with }u^{1}(0)+\Lambda_{1}=0\mbox{ and }u^{2}(0)+\Lambda_{2}>0.
    \end{equation*}
    Then, we have
    \begin{equation*}
        W_{1}(\mathbf{u}(0))\leq\int_{0}^{1}W_{12}\Big(-\mathbf{e}+(\mathbf{u}(0)+\mathbf{e})\cdot t\Big)\cdot(u^{2}(0)+\Lambda_{2})dt<0.
    \end{equation*}
    This contradicts the previously obtained inequality
\[
W_1(\mathbf u(0))=\Delta_pu^1(0)\geq0.
\]
Hence, the proof is complete.
\end{proof}

\begin{lemma}\label{le5.3}
    Let $\mathbf{u}(x)$ be a $C^{1,\epsilon}_{loc}$ solution to \eqref{eq. GP equation} satisfying \eqref{eq. speed identity} near $x=0$. Assume that $\mathbf{u}(0)\neq\pm\mathbf{e}$ and that $\mathbf{u}$ is pointwisely increasing at $x=0$, i.e.:
    \begin{equation*}
        u^{i}(-x)\leq u^{i}(0)\leq u^{i}(x),\quad\mbox{for all }i\in\{1,\cdots,m\}\mbox{ and sufficiently small }x\geq0.
    \end{equation*}
    Then $u^{i}_{x}(0)>0$ for all $i\in\{1,\cdots,m\}$.
\end{lemma}
\begin{proof}
By the monotonicity assumption,
\[
u_x^i(0)\geq0,
\qquad\mbox{for all }i\in\{1,\cdots,m\}.
\]
    By \eqref{eq. speed identity}, there exists at least one $i\in\{1,\cdots,m\}$, such that
    \begin{equation*}
        u^{i}_{x}(0)\geq c_{0}:=\Big(\frac{p}{(p-1)\cdot m}\cdot W\big(\mathbf{u}(0)\big)\Big)^{1/p}.
    \end{equation*}
Without loss of generality, assume that
\[
u_x^1(0)\geq c_0 .
\]
It suffices to prove that
\[
u_x^i(0)>0,\qquad i\geq2 .
\]

    Suppose that we instead have $u^{2}_{x}(0)=0$ (i.e., we consider $i=2$). Fix a $c_{1}>0$ (to be specified later), such that $c_{1}/c_{0}$ is sufficiently small. By the $C^{\epsilon}$ continuity of $\mathbf{u}_{x}$, for all sufficiently small $x_{0}>0$ and $x\in[-x_{0},x_{0}]$, one has
    \begin{equation}\label{eq. severe derivative difference}
        u^{1}_{x}(x)\geq\frac{c_{0}}{2},\quad|u^{2}_{x}(x)|\leq c_{1},\quad\mbox{and }u^{i}_{x}\geq-c_{1}\quad\mbox{for all }i\in\{1,\cdots,m\}.
    \end{equation}
    Let $M$ be a uniform constant depending on the potential function $W(\cdot)$, such that
    \begin{equation}\label{eq. very negative mixed derivative}
        |D_{2j}W(\xi)|\leq-M\cdot D_{21}W(\xi),\quad\mbox{for all }\xi\in[-\Lambda_{1},\Lambda_{1}]\times\cdots\times[-\Lambda_{m},\Lambda_{m}]\mbox{ and }j\neq1.
    \end{equation}
    The existence of $M$ is guaranteed by the assumptions (A1)-(A5).
    
    Let $\mathbf{v}(x)=\mathbf{u}(x)-\mathbf{u}(-x)$, then $\mathbf{v}(x)$ (with $x\in[0,x_{0}]$) satisfies
    \begin{equation*}
        \Delta_{p}v^{2}(x)=D_{2}W(\mathbf{u}(x))-D_{2}W(\mathbf{u}(-x)).
    \end{equation*}
    Using the Lagrange mean value theorem, there exists a point $\xi=(\xi^{1},\cdots,\xi^{m})$ lying on the line segment $[\mathbf{u}(-x),\mathbf{u}(x)]$, such that
    \begin{align*}
        \Delta_{p}v^{2}(x)=&\sum_{j=1}^{m}D_{2j}W(\xi)\cdot(u^{j}(x)-u^{j}(-x))\\
        \leq&-|D_{21}W(\xi)|\cdot c_{0}x+|D_{22}W(\xi)|\cdot2c_{1}x-\sum_{j\geq2}D_{2j}W(\xi)\cdot2c_{1}x.
    \end{align*}
    where we have used \eqref{eq. severe derivative difference} and that $D_{2j}W<0$ for $j\neq2$. By \eqref{eq. very negative mixed derivative}, one has
    \begin{equation*}
        \Delta_{p}v^{2}(x)\leq-|D_{21}W(\xi)|\cdot c_{0}x+m\cdot M|D_{21}W(\xi)|\cdot c_{1}x\leq-\frac{1}{2}|D_{21}W(\xi)|\cdot c_{0}x,
    \end{equation*}
    provided that $\displaystyle c_{1}:=\frac{c_{0}}{2m\cdot M}$. By denoting $c_{2}:=\frac{1}{2}|D_{21}W(\xi)|\cdot c_{0}$, we have
    \begin{equation*}
        \Delta_{p}v^{2}(x)\leq-c_{2}x,\quad\mbox{for }x\in[0,x_{0}].
    \end{equation*}
    Integrating such an inequality over $[0,x_{0}]$ yields:
    \begin{equation*}
        |u^{2}_{x}(x_{0})|^{p-2}u^{2}_{x}(x_{0})+|u^{2}_{x}(-x_{0})|^{p-2}u^{2}_{x}(-x_{0})-2|u^{2}_{x}(0)|^{p-2}u^{2}_{x}(0)\leq-\frac{c_{2}}{2}x_{0}^{2}.
    \end{equation*}
    
    Now we can derive a contradiction. If $u^{2}_{x}(0)=0$, then
    \begin{equation*}
        |u^{2}_{x}(x_{0})|^{p-2}u^{2}_{x}(x_{0})+|u^{2}_{x}(-x_{0})|^{p-2}u^{2}_{x}(-x_{0})<-\frac{c_{2}}{2}x_{0}^{2},
    \end{equation*}
    which further implies that
    \begin{equation*}
        u^{2}_{x}(x_{0})+u^{2}_{x}(-x_{0})<0\mbox{ for }x_{0}\mbox{ sufficiently small}.
    \end{equation*}
    However, integrating such an inequality (by de-freezing $x_{0}$) implies $$u^{2}(x)<u^{2}(-x)$$ for all sufficiently small $x>0$, which contradicts the assumption that $\mathbf{u}$ is increasing at $x=0$.
\end{proof}
\subsection{On the sliding method}
We now prove the non-strict monotonicity of the heteroclinical solution $\mathbf{u}$ using the sliding method, which is effective for studying equations with phase-transition background (see \cite{BN1991} for $p=2$ and $m=1$).

\begin{lemma}[Monotonicity]\label{le5.4}
Let $\mathbf{u}(x)$ be a heteroclinical solution of \eqref{eq. GP equation} with $\mathbf{u}(\pm\infty)=\pm\mathbf{e}$. Then $u^{i}(x)\geq 0$ for $i\in\{1,\cdots,m\}$.
\end{lemma}
\begin{proof}
By Definition~\ref{def. heteroclinical}, there exist $a,b\in[-\infty,+\infty]$, such that
\begin{equation*}
    (a,b)=\{x\in\mathbb{R}:\mathbf{u}(x)\neq\pm\mathbf{e}\}.
\end{equation*}
By Corollary \ref{cor3.1}, we deduce that $u^i$ is strictly monotonic near $a+$ and $b-$ for any $1\leq i\leq m$. Therefore, there exist $a'$ and $b'$ such that
\begin{itemize}
    \item[(i)] $|a'|, |b'|<+\infty$, and $a<a'<b'<b$;
    \item[(ii)] $\mathbf{u}(x)$ is sufficiently close to $-\mathbf{e}$ (resp. $+\mathbf{e}$) when $x\in(a,a']$ (resp. $x\in[b',b)$);
    \item[(iii)] $u^i(x)$'s are strictly increasing on the intervals $(a, a')$ and $(b', b)$.
\end{itemize}
Since $\mathbf{u}(x)\neq\pm\mathbf{e}$ in $[a',b']$, it then follows from Lemma \ref{lem. boundary stability of pm e} that
\begin{equation*}
    u^i(x)\neq\pm \Lambda_i,\quad\mbox{for any }x\in[a',b']\mbox{ and }i\in\{1,\cdots,m\}.
\end{equation*}
Then, let us denote 
\begin{equation*}
    d=\inf_{\substack{x\in[a',b']\\i\in\{1,\cdots,m\}}}\Big(\min\big\{\Lambda_{i}-u^{i}(x), u^{i}(x)+\Lambda_{i}\big\}\Big)>0.
\end{equation*}
Since $\mathbf{u}(\pm\infty)=\pm\mathbf{e}$, there exist $a<a''\leq a'$ and $b'\leq b''<b$, such that
\begin{equation*}
    \begin{cases}
\Lambda_i-u^i(x)<d, \,\, &\mbox{if}\,\, x> b'',\\
u^i(x)+\Lambda_i<d, \,\, &\mbox{if}\,\, x<a'',
\end{cases}\quad\mbox{for all }i\in\{1,\cdots,m\}.
\end{equation*}
Moreover, recalling that $\mathbf{u}(x)\equiv-\mathbf{e}$ when $x\leq a$ and $\mathbf{u}(x)\equiv\mathbf{e}$ when $x\geq b$, we then have that $u^{i}(x)$'s are (non-strictly) increasing on the intervals $(-\infty, a'']$ and $[b'', +\infty)$. In conclusion, for all $x_{1}\leq x_{2}\leq a''\leq y\leq b''\leq z_{1}\leq z_{2}$, one has
\begin{equation}\label{M0L}
    u^{i}(x_{1})\leq u^{i}(x_{2})\leq u^{i}(y)\leq u^{i}(z_{1})\leq u^{i}(z_{2}),\quad\mbox{for all }i\in\{1,\cdots,m\}.
\end{equation}

In order to show the monotonicity of $u^{i}$,  we introduce the sliding argument.
For every $\tau\geq0$,  we define
\begin{equation*}
    x^\tau=x+\tau,\quad\mathbf{u}_{\tau}(x)=\mathbf{u}(x^{\tau}),\quad\mathbf{w}_\tau (x)=\mathbf{u}_\tau (x)-\mathbf{u}(x).
\end{equation*}
Our goal is to show that 
$$w_{\tau}^{i}(x)\geq0\,\, \mbox{for all}\,\, x\in\mathbb{R}, \,\tau\geq0, \mbox{and}\,\, i\in\{1,\cdots,m\}.$$
    
\textbf{Step 1.} Let $$M_{0}=|a''|+|b''|.$$ We first show that
\begin{equation}\label{step1-1}
w^{i}_{\tau}(x) \geq 0  \,\, \mbox{in}\,\, \mathbb{R},\,\, \mbox{for}\,\, \tau >2M_0\,\, \mbox{and each}\,\,i\in\{1,\cdots,m\},
\end{equation}
where $w^{i}_\tau(x)$ is the $i$-th component of $w_{\tau}(x)$. Notice that for $\tau>2M_{0}$, one has that the interval $[a''-\tau,b''-\tau]$ lies totally to the left of the interval $[a'',b'']$. To show \eqref{step1-1}, we consider the following two cases:
\begin{itemize}
    \item Case 1: $x\geq a''$. In this case, one notice that $x^{\tau}\geq b''$, then \eqref{step1-1} follows from the 3rd or 4th inequality of \eqref{M0L}.
    \item Case 2: $x\leq a''$. In this case, as $x<x^{\tau}$, \eqref{step1-1} follows from the 1st to 3rd inequality of \eqref{M0L}.
\end{itemize}

 {\bf{Step 2.}} Step 1 provides a starting point, from which we can carry out the sliding. From $\tau = 2M_0$, we decrease $\tau,$  and show that for any $0 <\tau< 2M_0$, we also have 
 \begin{equation}\label{step2}
w^{i}_{\tau}(x) \geq 0  \,\, \mbox{in}\,\, \mathbb{R},\,\, \mbox{for}\,\, 0<\tau <2M_0\,\, \mbox{and for each}\,\,i\in\{1,\cdots,m\},
\end{equation}
Define 
$$
\tau_0=\inf\{\tau \mid w^{i}_{\tau}(x) \geq  0,\,\, i=1,\cdots, m,\,\, x \in \mathbb{R},\,\, 0<\tau<2M_0.\}
$$
By definition and continuity of $u$, we have
\begin{equation}\label{eq. positivity by continuity}
w^i_{\tau_0}(x)\ge 0 \quad \text{in } \mathbb{R}.
\end{equation}

We make the following claim:
\begin{itemize}
    \item \textbf{Claim:} There exists an $x_{0}\in\mathbb{R}$, such that 
    \[
\mathbf u_{\tau_0}(x_0)=\mathbf u(x_0)\neq\pm\mathbf e,
\]
and
\[
u_{\tau_0,x}^i(x_0)=u_x^i(x_0)>0,
\qquad i=1,\ldots,m .
\]
\end{itemize}

The proof of the claim will be postponed to Step 3. Assuming the claim for
the moment, we continue the proof of Step 2 and show that
\[
\tau_0=0.
\]

Let $$A(s)=|s|^{p-2}s.$$ Each component $u^i(x)$ satisfies $$\displaystyle\frac{d}{dx}A(u^{i}_{x})=W_{i}(\mathbf{u}).$$ Since $\mathbf{u}_{\tau_0}(x)$ is also a solution of  the same system, subtracting the two equations yields
\begin{equation*}
    \frac{d}{dx}\Big(A\big(u^{i}_{x}+(w_{\tau_0}^{i})_{x}\big)-A(u^{i}_{x})\Big)=W_{i}(\mathbf{u}+\mathbf{w}_{\tau_0})-W_{i}(\mathbf{u})=\sum_{j=1}^{m}c_{ij}(x)w_{\tau_0}^{j}.
\end{equation*}
By the assumption $(A1)$ on the potential $W(\cdot)$, we have 
\begin{equation}\label{eq. c_ij for w_tau is cooperative}
    c_{ij}(x)<0,\quad\mbox{for all }x\mbox{ close to }x_{0}\mbox{ and }i\in\{1,\cdots,m\},
\end{equation}
where
\[
c_{ij}(x)
=
\int_0^1
W_{ij}\big(\mathbf u+t\mathbf w_{\tau_0}\big)\,dt .
\]
By the mean value theorem, there exist $\rho^{i}=\rho^{i}(x)\in(0,1)$ for $i\in\{1,\cdots,m\}$ such that
\begin{align*}
    &\frac{d}{dx}\Big(A\big(u^{i}_{x}+(w_{\tau_0}^{i})_{x}\big)-A(u^{i}_{x})\Big)=\frac{d}{dx}\Big(A'\big(u^{i}_{x}+\rho\cdot(w_{\tau_0}^{i})_{x}\big)\cdot(w_{\tau_0}^{i})_{x}\Big).
\end{align*}
For each $i\in\{1,\cdots,m\}$, set
\begin{equation*}
    a^{i}(x)=A'\big(u^{i}_{x}+\rho^{i}\cdot(w_{\tau_0}^{i})_{x}\big),\quad b^{i}(x)=\frac{d}{dx}\Big(A'\big(u^{i}_{x}+\rho^{i}\cdot(w_{\tau_0}^{i})_{x}\big)\Big),
\end{equation*}
then $\mathbf{w}_{\tau_0}$ satisfies a linearized equation in the following divergence form:
\begin{equation}\label{uniformeq}
    \frac{d}{dx}\Big(a^{i}(x)\cdot(w^{i}_{\tau_0})_{x}\Big)=\sum_{j=1}^{m}c_{ij}(x)w^{j}_{\tau_0}.
\end{equation}

By the $C^{1,\epsilon}$ regularity of $\mathbf{u}$ (as well as $\mathbf{u}_{\tau}$), it then follows from the \textbf{Claim} that
\begin{equation*}
    u^{i}_{x},(u_{\tau_0}^{i})_{x}>0,\quad\mbox{for all }x\mbox{ close to }x_{0}\mbox{ and }i\in\{1,\cdots,m\}.
\end{equation*}
Since $A(s)=|s|^{p-2}s$ has strictly positive derivative for $s>0$, it follows that
\begin{equation}\label{eq. a^i(x) is uniform}
    0<c_{1}\leq a^{i}(x)\leq C_{2}<\infty,\quad\mbox{for all }x\mbox{ close to }x_{0}\mbox{ and }i\in\{1,\cdots,m\},
\end{equation}
By \eqref{eq. c_ij for w_tau is cooperative} and \eqref{eq. a^i(x) is uniform}, it follows that the linearized equation \eqref{uniformeq} is uniformly elliptic and strictly cooperative near $x_{0}$. By \eqref{eq. positivity by continuity} and the \textbf{Claim}, applying the strong maximum principle for cooperative systems yields that
\begin{equation}\label{eq. totally contact}
    \mathbf{u}\equiv\mathbf{u}_{\tau_0},\quad\mbox{for all }x\mbox{ sufficiently close to }x_{0}.
\end{equation}
Then, every $x_{1}$ sufficiently close to $x_{0}$ is a contact point between $\mathbf{u}_{\tau}$ and $\mathbf{u}$. As long as $\mathbf{u}(x_{1})\neq\pm\mathbf{e}$, using the same argument as in \textbf{Step 3} below shows that $x_{1}$ satisfies all requirements as given in the \textbf{Claim} (we leave the verification job to the readers). Repeating the argument above yields that \eqref{eq. totally contact} also holds at $x_{1}$. Then an open-and-closed argument yields that $\mathbf{u}\equiv\mathbf{u}_{\tau_0}$ everywhere, and thus $\tau_{0}=0$.

By combining \textbf{Step 1} and \textbf{Step 2}, we complete the proof of Lemma \ref{le5.4}.

\textbf{Step 3.} Finally, let us prove the \textbf{Claim} given previously in Step 2. The proof is broken into three sub-steps.
\begin{itemize}
    \item First, we show that there exist an $x_{0}\in\mathbb{R}$ and an $i\in\{1,\cdots,m\}$, such that
    \begin{equation*}
        u_{\tau_{0}}^{i}(x_{0})=u^{i}(x_{0})\quad\mbox{and }\mathbf{u}(x_{0})\neq\pm\mathbf{e}.
    \end{equation*}
    
    To see this, we first obtain from the supremum assumption on $\tau_{0}$ that there exists a sequence $\{x_{k}\}\subseteq\mathbb{R}$ for sufficiently large $k$, such that
    \begin{equation}\label{eq. w(x_k) negative}
        w_{\tau_{0}-\frac{1}{k}}^{i_{k}}(x_{k})<0\quad\mbox{for some }i_{k}\in\{1,\cdots,m\}.
    \end{equation}
    However, let $x\in(-\infty,a'')\cup(b'',+\infty)$ and let $\tau\geq\frac{\tau_{0}}{2}$, then it follows from $x^{\tau}\geq x$ and \eqref{M0L} that $$\mathbf{u}_{\tau}(x)\geq \mathbf{u}(x).$$ Therefore, we see that $x_{k}\in[a'',b'']$ for sufficiently large $k$. Consequently, there exists a subsequence of $x_{k}$ and some $x_{\infty}\in[a'',b'']$, such that $x_{k}\to x_{\infty}$ and $i_{k}=constant=:i_{\infty}$. It then follows from \eqref{eq. positivity by continuity} and \eqref{eq. w(x_k) negative} that $$w_{\tau_{0}}^{i_{\infty}}(x_{\infty})=0.$$ Moreover, since $x_{\infty}\in[a'',b'']$, we get $\mathbf{u}(x_{\infty})\neq\pm\mathbf{e}$. Then the first sub-step is verified by setting $x_{0}=x_{\infty}$ and $i=i_{\infty}$.

    \item Second, we further show that $$\mathbf{u}_{\tau_{0}}(x_{0})=\mathbf{u}(x_{0}).$$ This means $\mathbf{u}_{\tau_{0}}$ touches $\mathbf{u}$ from above at $x_{0}$ in all components, which further implies that $$\nabla\mathbf{u}_{\tau_{0}}(x_{0})=\nabla\mathbf{u}(x_{0}).$$
    
    We show $\mathbf{u}_{\tau_{0}}(x_{0})=\mathbf{u}(x_{0})$ by contradiction. Without loss of generality, assume that $i=1$ in the conclusion of the first sub-step, and suppose that $u_{\tau_{0}}^{2}(x_{0})\neq u^{2}(x_{0})$. By \eqref{eq. positivity by continuity}, the following inequalities hold in a small neighborhood of $x_{0}$:
    \begin{equation*}
        u_{\tau_{0}}^{j}(x)\geq u^{j}(x)\mbox{ for all }j\in\{i,\cdots,m\},\quad u_{\tau_{0}}^{2}(x)>u^{2}(x).
    \end{equation*}
    By the assumption (A1) on $W(\cdot)$ and by $u_{\tau_{0}}^{1}(x_{0})>u^{1}(x_{0})$, we have
    \begin{equation*}
        D_{1}W\big(\mathbf{u}_{\tau}(x_{0})\big)<D_{1}W\big(\mathbf{u}(x_{0})\big).
    \end{equation*}
    This further implies that $$\Delta_{p}u_{\tau_{0}}^{1}(x)<\Delta_{p}u^{1}(x)$$ in a small neighborhood of $x_{0}$, which contradicts the fact that $u_{\tau_{0}}^{1}$ touches $u^{1}$ from above at $x_{0}$.

    \item Third, we show that $\mathbf{u}$ is pointwisely increasing at $x_{0}$. Precisely speaking, we intend to show that for all $t>0$ and $i\in\{1,\cdots,m\}$,
    \begin{equation*}
        u^{i}(x_{0}-t)\leq u^{i}(x_{0})\leq u^{i}(x_{0}+t).
    \end{equation*}

    To see this, suppose that it holds for some $t>0$ and $i\in\{1,\cdots,m\}$ that
    \begin{equation*}
        u^{i}(x_{0})>u^{i}(x_{0}+t)\mbox{ or }u^{i}(x_{0})<u^{i}(x_{0}-t).
    \end{equation*}
    Consider $\tau=\tau_{0}+t>\tau_{0}$, then we must have $u_{\tau}^{i}\geq u^{i}$ in $\mathbb{R}$. On the other hand, if $u^{i}(x_{0})>u^{i}(x_{0}+t)$ or $u^{i}(x_{0})<u^{i}(x_{0}-t)$ holds, then
    \begin{equation*}
        u_{\tau}^{i}(x_{0})=u^{i}(x_{0}+t)<u^{i}(x_{0})\mbox{ or }u_{\tau}^{i}(x_{0}-t)=u^{i}(x_{0})<u^{i}(x_{0}-t),
    \end{equation*}
    causing a contradiction.
\end{itemize}
In the end, we complete the proof of the \textbf{Claim} by applying Lemma~\ref{le5.3}.
\end{proof}

\subsection{Proof of Theorem~\ref{thm. basic}}
We end this section by giving the proof of Theorem~\ref{thm. basic}.
\begin{proof}[Proof of Theorem~\ref{thm. basic}]
The existence of a heteroclinical solution to \eqref{eq. GP equation} satisfying \eqref{eq. speed identity} follows from the variational construction in Lemma~\ref{lem. existence}.

The monotonicity of $\mathbf{u}$ and the comparability of all components of $\mathbf{u}_{x}$ (i.e. $u^{i}_{x}\sim u^{j}_{x}$) near $a+$ or near $b-$ is a direct consequence of Corollary~\ref{cor3.1}. The monotonicity of $\mathbf{u}$ and the comparability of all components of $\mathbf{u}_{x}$ away from $a+$ and $b-$ is a consequence of Lemma~\ref{le5.3} and Lemma~\ref{le5.4}.

Finally, the asymptotic estimates for $u^{i}(x)+\Lambda_{i}$ follows from Lemma~\ref{lem. (non-)existence of free boundary} together with the fact $u^{i}_{x}\sim u^{j}_{x}$.
\end{proof}

\section{Curvature of the trajectory}
In this section, we investigate the geometric aspect of the heteroclinical solution and prove Theorem~\ref{thm. main}.
\subsection{An indirect way of computing the curvature}
Let $$\gamma(s)=(\gamma^{1}(s),\cdots,\gamma^{m}(s))$$ be a reparameterization of $\mathbf{u}(x)$, in the sense that $s=s(x)$ is an increasing function, and $\mathbf{u}(x)=\gamma(s(x))$. Moreover, we require that $\gamma(s)$ is arc-length in the $l_{p}$ sense, i.e.:
\begin{equation}\label{eq. unit lp curve}
    |\gamma'(s)|_{l_{p}}=\Big(\sum_{i=1}^{m}|\gamma^{i}_{s}|^{p}\Big)^{1/p}=1.
\end{equation}
Define
\begin{equation*}
    \widetilde{W}(s)=W\Big(\mathbf{u}(x(s))\Big).
\end{equation*}
As $|\mathbf{u}_{x}|_{l_{p}}=\Big(\frac{p}{p-1}\cdot W(\mathbf{u})\Big)^{1/p}=\frac{ds}{dx}$, we have
\begin{equation*}
    \mathbf{u}_{x}=\Big(\frac{p}{p-1}\cdot\widetilde{W}(s)\Big)^{1/p}\gamma_{s}=\frac{ds}{dx}\cdot\gamma_{s},\quad\mbox{i.e. }u^{i}_{x}=\Big(\frac{p}{p-1}\cdot\widetilde{W}(s)\Big)^{1/p}\gamma^{i}_{s}=\frac{ds}{dx}\cdot\gamma^{i}_{s}.
\end{equation*}
Besides,
\begin{equation*}
    \frac{ds}{dx}=\Big(\frac{p}{p-1}\cdot\widetilde{W}(s)\Big)^{1/p}
\end{equation*}

Let us take the second derivative of $\gamma$, then
\begin{equation}\label{eq. p-perpendicular}
    p\sum_{i=1}^{m}|\gamma^{i}_{s}|^{p-2}\gamma^{i}_{s}\cdot\gamma^{i}_{ss}=\frac{d}{ds}\Big(\sum_{i=1}^{m}|\gamma^{i}_{s}|^{p}\Big)=0.
\end{equation}

The following lemma is the key ingredient in this subsection, which provides an indirect way of computing the curvature.
\begin{lemma}
    We have the following identity:
    \begin{equation}\label{eq. how to calculate the magnetude of the curvature}
        \sum_{i=1}^{m}W_{i}(\mathbf{u})\gamma^{i}_{ss}=p\cdot\widetilde{W}(s)\sum_{i=1}^{m}|\gamma^{i}_{s}|^{p-2}|\gamma^{i}_{ss}|^{2}.
    \end{equation}
\end{lemma}
\begin{proof}
    For simplicity, denote $$v(s)=\Big(\frac{p}{p-1}\cdot\widetilde{W}(s)\Big)^{1/p}.$$  Using the chain rule, we obtain
    \begin{equation*}
        u^{i}_{xx}=\frac{d}{dx}(\frac{ds}{dx}\cdot\gamma^{i}_{s})=\frac{ds}{dx}\cdot\frac{d}{ds}(v(s)\cdot\gamma^{i}_{s})=v(s)v'(s)\gamma^{i}_{s}+v(s)^{2}\gamma^{i}_{ss}.
    \end{equation*}
    Then, \eqref{eq. GP equation} implies
    \begin{equation*}
        W_{i}(\mathbf{u})=(p-1)|u^{i}_{x}|^{p-2}u^{i}_{xx}=(p-1)|u^{i}_{x}|^{p-2}\cdot\Big\{v(s)v'(s)\gamma^{i}_{s}+v(s)^{2}\gamma^{i}_{ss}\Big\}.
    \end{equation*}
    Since $u^{i}_{x}=v(s)\cdot\gamma^{i}_{s}$, we then have
    \begin{equation*}
        W_{i}(\mathbf{u})=(p-1)v(s)^{p-2}|\gamma^{i}_{s}|^{p-2}\cdot\Big\{v(s)v'(s)\gamma^{i}_{s}+v(s)^{2}\gamma^{i}_{ss}\Big\}.
    \end{equation*}
    Multiplying by $\gamma^{i}_{ss}$ on both sides, and summing over $i$ gives
    \begin{equation*}
        \sum_{i=1}^{m}W_{i}(\mathbf{u})\gamma^{i}_{ss}=(p-1)v(s)^{p-1}v(s)'\sum_{i=1}^{m}|\gamma^{i}_{s}|^{p-2}\gamma^{i}_{s}\gamma^{i}_{ss}+(p-1)v(s)^{p}\sum_{i=1}^{m}|\gamma^{i}_{s}|^{p-2}|\gamma^{i}_{ss}|^{2}.
    \end{equation*}
    Using \eqref{eq. p-perpendicular}, we have
    \begin{equation*}
        \sum_{i=1}^{m}W_{i}(\mathbf{u})\gamma^{i}_{ss}=(p-1)v(s)^{p}\sum_{i=1}^{m}|\gamma^{i}_{s}|^{p-2}|\gamma^{i}_{ss}|^{2}.
    \end{equation*}
    Then, we have verified \eqref{eq. how to calculate the magnetude of the curvature}.
\end{proof}
\subsection{Proof of Theorem~\ref{thm. main}}
With the formulas \eqref{eq. p-perpendicular} and \eqref{eq. how to calculate the magnetude of the curvature}, we have a indirect way of estimating the magnitude to $\gamma_{ss}$, which leads to the proof of Theorem~\ref{thm. main}.
\begin{proof}[Proof of Theorem~\ref{thm. main}]
Recall that Corollary~\ref{cor3.1} implies
\[
u^{i}_{x}\sim u^{j}_{x}
\]
as $x\to a+$ or $x\to b-$. Moreover, it follows from Lemma~\ref{le5.3} and Lemma~\ref{le5.4} that $u^{i}_{x}\sim u^{j}_{x}$ when $x$ is away from $a$ and $b$. Therefore, part (1) of Theorem~\ref{thm. main} follows from the observation that $\gamma_{s}$ is parallel to $\mathbf{u}_{x}$ when $s=s(x)$.

We next prove part (2). Recall that it follows from Corollary~\ref{cor3.1} that $DW(\mathbf{u})$ is ``almost parallel to" $\big(|u^{1}_{x}|^{p-2}u^{1}_{x},\cdots,|u^{m}_{x}|^{p-2}u^{m}_{x}\big)$ as $x\to a+$ or $x\to b-$. Define
\begin{equation*}
    \Theta(x)=\angle{\Big(\big(|u^{1}_{x}|^{p-2}u^{1}_{x},\cdots,|u^{m}_{x}|^{p-2}u^{m}_{x}\big),\big(W_{1}(\mathbf{u}),\cdots,W_{m}(\mathbf{u})\big)\Big)},
\end{equation*}
one can infer from Corollary~\ref{cor3.1} that
\begin{equation}\label{eq. Theta to 0}
    \lim_{x\to a+}\Theta(x)=\lim_{x\to b-}\Theta(x)=0.
\end{equation}

By the definition of $\gamma(s)$, and by \eqref{eq. p-perpendicular}, one has
\begin{equation*}
    \angle{\Big(\big(\gamma^{1}_{ss},\cdots,\gamma^{m}_{ss}\big),\big(W_{1}(\mathbf{u}),\cdots,W_{m}(\mathbf{u})\big)\Big)}=\frac{\pi}{2}\pm\Theta.
\end{equation*}
Consequently,
\begin{equation}\label{eq. cosine rule}
    \Big|\sum_{i=1}^{m}W_{i}(\mathbf{u})\gamma^{i}_{ss}\Big|=|DW(\mathbf{u})|_{l_{2}}\cdot|\gamma_{ss}|_{l_{2}}\cdot\sin{(\Theta)}.
\end{equation}

Moreover, as we have shown $\gamma_{s}^{i}\sim\gamma_{s}^{j}$ in Theorem~\ref{thm. main} (1), it follows from \eqref{eq. unit lp curve} that
\begin{equation}\label{eq. rhs of the indirect formula}
    p\cdot\widetilde{W}(s)\sum_{i=1}^{m}|\gamma^{i}_{s}|^{p-2}|\gamma^{i}_{ss}|^{2}\geq c\cdot\widetilde{W}(s)\cdot|\gamma_{ss}|_{l_{2}}^{2}.
\end{equation}
Combining \eqref{eq. how to calculate the magnetude of the curvature} with the additional information \eqref{eq. cosine rule}-\eqref{eq. rhs of the indirect formula} yields that
\begin{equation*}
    |DW(\mathbf{u})|_{l_{2}}\cdot\sin{(\Theta)}\geq c\cdot\widetilde{W}(s)\cdot|\gamma_{ss}|_{l_{2}}=c\cdot W\big(\mathbf{u}(s)\big)\cdot|\gamma_{ss}|_{l_{2}}.
\end{equation*}
By \eqref{eq. Theta to 0} together with the homogeneity assumptions (A4)-(A5), we conclude that
\begin{equation*}
    |\gamma_{ss}|_{l_{2}}\leq C\cdot\frac{|DW(\mathbf{u})|_{l_{2}}}{W\big(\mathbf{u}(s)\big)}\cdot\sin{(\Theta)}=\left\{\begin{aligned}
        &o\big(\frac{1}{s-s_{-}}\big),&\mbox{as }&s\to s_{-},\\
        &o(\frac{1}{s_{+}-s}),&\mbox{as }&s\to s_{+}.
    \end{aligned}\right.
\end{equation*}
Then, the desired estimate in part (2) follows as $|\gamma_{s}(s)|\sim|\gamma_{s}(s)|_{l_{p}}=1$.
\end{proof}

\section{A casual talk of possible application}\label{sec. casual talk}
Finally, we briefly discuss a possible application of the above analysis and explain our motivation for studying this problem. This section can also be viewed as a continuation of the discussion in the Introduction.

We believe that a similar construction can be extended to the Bose-Einstein condensation case (i.e., $m\geq2$). However, the key difference is that a curvature term pops out, as the trajectory of the heteroclinical solution in a higher dimensional region \[
[-\Lambda_{1},\Lambda_{1}]\times\cdots\times[-\Lambda_{m},\Lambda_{m}]
\] is generally not a straight line.

Consider the following reparameterization of $\mathbf{u}(x)$:
\begin{equation*}
    \mathbf{g}(t)=\gamma(s(t)),\quad\mbox{where }\frac{ds}{dt}=\sqrt[p]{\frac{p}{p-1}\cdot h(s)}.
\end{equation*}
Let us compute $\frac{d}{dt}\mathbf{g}$:
\begin{align*}
    \frac{d}{dt}\mathbf{g}=\frac{ds}{dt}\cdot\frac{d}{ds}\gamma(s)=\sqrt[p]{\frac{p}{p-1}\cdot h(s)}\cdot\gamma_{s}.
\end{align*}
Let us compute $\Delta_{p}\mathbf{g}(t):=\frac{d}{dt}\Big(|g^{i}_{t}|^{p-2}\cdot g^{i}_{t}\Big)$ under the further assumption that $h'(s)>0$ and $\gamma^{i}_{s}>0$:
\begin{align*}
    \Delta_{p}\mathbf{g}(t)=&\frac{d}{dt}\Big\{\Big(\frac{p}{p-1}\cdot h(s)\Big)^{\frac{p-1}{p}}\cdot(\gamma^{i}_{s})^{p-1}\Big\}\\
    =&\sqrt[p]{\frac{p}{p-1}\cdot h(s)}\cdot\frac{d}{ds}\Big\{\Big(\frac{p}{p-1}\cdot h(s)\Big)^{\frac{p-1}{p}}\cdot(\gamma^{i}_{s})^{p-1}\Big\}\\
    =&h'(s)\cdot(\gamma^{i}_{s})^{p-1}+p\cdot h(s)\cdot(\gamma^{i}_{s})^{p-2}\cdot\gamma^{i}_{ss}.
\end{align*}
For the same reason, we have
\begin{equation*}
    \nabla_{i}W(\mathbf{u})=\Delta_{p}u^{i}(x)=\widetilde{W}'(s)\cdot(\gamma^{i}_{s})^{p-1}+p\cdot\widetilde{W}(s)\cdot(\gamma^{i}_{s})^{p-2}\cdot\gamma^{i}_{ss}.
\end{equation*}

Let us assume that for some positive and increasing function $\mathcal{E}(s)$,
\begin{equation*}
    h(s)=\widetilde{W}(s)-\mathcal{E}(s).
\end{equation*}
Then, we obtain the following identity:
\begin{equation*}
    \nabla_{i}W(\mathbf{g})-\Delta_{p}g^{i}(t)=\mathcal{E}'(s)\cdot(\gamma^{i}_{s})^{p-1}+p\cdot\mathcal{E}(s)\cdot(\gamma^{i}_{s})^{p-2}\cdot\gamma^{i}_{ss}.
\end{equation*}

Let us consider a weaker form of \eqref{eq. radial barrier supersolution}. More precisely, when $R$ is sufficiently large, \eqref{eq. radial barrier supersolution} is ``almost equivalent" to
\begin{equation*}
    \Delta_{p}g^{i}(t)=\frac{d}{dt}\Big(|g^{i}_{t}|^{p-2}\cdot g^{i}_{t}\Big)<\nabla_{i}W(\mathbf{g}).
\end{equation*}
From the discussion above, it suffices to require
\begin{equation}\label{eq. why curvature is important}
    \mathcal{E}'(s)\cdot\gamma^{i}_{s}+p\cdot\mathcal{E}(s)\cdot\gamma^{i}_{ss}>0.
\end{equation}
When $\mathcal{E}'(s)>0$ and $m=1$, the inequality above is immediate since  $$ \gamma_{ss}\equiv0.$$ 
This corresponds to the classical construction of super-solutions for the Allen-Cahn equation. 

However, in the Bose-Einstein condensation setting with $m\geq2$, the curvature term $\gamma^i_{ss}$ is generally nonzero and makes the verification of \eqref{eq. why curvature is important} considerably more delicate. This observation motivates the curvature estimate established in Theorem~\ref{thm. main}.

Let $[s_{-},s_{+}]$ be the domain of the arc-length re-parametrization, as assumed in Theorem~\ref{thm. main}. By choosing different $\mathcal{E}(s)$, one can verify \eqref{eq. why curvature is important} in some sub-intervals of $[s_{-},s_{+}]$. We list two important examples below.
\begin{itemize}
    \item[(1)] When we choose $$\mathcal{E}(s)=c_{1}\cdot(s-s_{-})^{\gamma}$$ for some $\gamma>0$, then Theorem~\ref{thm. main} indicates that \eqref{eq. why curvature is important} holds when $s$ is sufficiently close to $s_{-}$.
    \item[(2)] When we choose $$\mathcal{E}(s)=c_{2}\cdot e^{M\cdot s},$$ and if $s$ is away from $s_{\pm}$, then $|\gamma_{ss}|$ is bounded from above, and also Theorem~\ref{thm. main} implies that $\gamma^{i}_{s}$'s for all $1\leq i\leq m$ have a positive lower bound. Therefore, we see that for sufficiently large $M$, \eqref{eq. why curvature is important} holds when $s$ is away from $s_{\pm}$.
\end{itemize}
Furthermore, the above two possible choices of $\mathcal{E}(s)$ can be joint in a $C^{1}$ differentiable way, which gives a supersolution to \eqref{eq. why curvature is important} for $s\in[s_{-},s_{+}-\epsilon]$. We believe that this will provide a systematic way of constructing suitable supersolutions to \eqref{eq. GP equation}.

\noindent \textbf{Acknowledgments} 
L. Wu is partially supported by National Natural Science Foundation of China (Grant No. 12401133) and the Guangdong Basic and Applied Basic Research Foundation (2025B151502069).

\vspace{2mm}

\noindent \textbf{Conflict of interest.} The authors do not have any possible conflicts of interest.

\vspace{2mm}

\noindent \textbf{Data availability statement.}
 Data sharing is not applicable to this article, as no data sets were generated or analyzed during the current study.

\bibliographystyle{abbrv}
\bibliography{main}
\end{document}